\documentclass[oneside,11pt,reqno]{amsart}
\usepackage{blkarray}
\usepackage{nicematrix}
\usepackage{amssymb}
\usepackage{graphicx,caption,subcaption} 
\usepackage{amsmath}
\usepackage{bbm}
\usepackage{amsthm} 
\usepackage{amsfonts}
\usepackage{setspace}
\usepackage{mathrsfs}
\usepackage{hyperref}
\usepackage{tensor}
\usepackage{tikz}
\usepackage{tikz-cd}
\usepackage{epstopdf}
\usepackage{enumerate}
\usepackage[enableskew]{youngtab}
\usepackage{young}
\usepackage[all,cmtip]{xy}
\usepackage[margin = .8 in]{geometry}
\usepackage{hyperref,color,polynom}

\usepackage{mdframed}
\usepackage{lipsum}

\newcommand{\zz}{\ensuremath{\mathbb{Z}}}
\newcommand{\qq}{\ensuremath{\mathbb{Q}}}
\newcommand{\nn}{\ensuremath{\mathbb{N}}}

\theoremstyle{definition}

\newmdtheoremenv{frm-thm}{Theorem}
\newmdtheoremenv{frm-def}{Definition}
\newmdtheoremenv{frm-lem}{Lemma}

\newtheorem{definition}{Definition}
\newtheorem{example}{Example}

\newtheorem{remark}{Remark}

\newtheorem{proof techniques}{Proof Techniques}

\newtheorem{lemma}{Lemma}
\newtheorem{corollary}{Corollary}

\newtheorem{proposition}{Proposition}

\newtheorem*{theorem*}{Theorem}
\newtheorem*{exercise*}{Exercise}
\newtheorem*{solution*}{Solution}

\newtheorem{lettertheorem}{Theorem}

\begin{document}

\title{Bruhat Intervals in the Infinite Symmetric Group are Cohen-Macaulay}
\author{Nathaniel Gallup and Leo Gray} 
\date{\today}

\begin{abstract}
    We show that the (non-Noetherian) Stanley-Reisner ring of the order complex of certain intervals in the Bruhat order on the infinite symmetric group $S_\infty$ of all auto-bijections of $\mathbb{N}$ is Cohen-Macaulay in the sense of ideals and weak Bourbaki unmixed. This gives an infinite-dimensional version of results due to Edelman, Bj\"{o}rner, and Kind and Kleinschmidt for finite symmetric groups $S_n$. 
\end{abstract}

\maketitle


\section{Introduction}

In commutative algebra, there has been much interest in developing combinatorial characterizations of the Cohen-Macaulay condition, particularly in the case of the Stanley-Reisner ring $k[\Delta]$ of a simplicial complex $\Delta$. For example, in their 1979 paper \cite{kind_schalbare_1979}, Kind and Kleinschmidt proved that if a finite pure simplicial complex $\Delta$ is shellable, then $k[\Delta]$ is Cohen-Macaulay, and in his 1972 work \cite{hochster1972}, Hochster showed that if such a $\Delta$ is constructible, then again $k[\Delta]$ is Cohen-Macaulay. An especially nice class of finite pure simplicial complexes are the order complexes of graded posets, so one might expect that the shellability criterion appears in a nice form in this context. Indeed, in his 1980 paper \cite{bjorner1980shellable}, Björner defined the notion of a lexicographically shellable poset and proved that their order complexes are shellable. In 1981, this was used by Edelman who proved in \cite{edelman1981bruhat} that the Bruhat order on (any interval in) $S_n$ is lexicographically shellable, thereby generating more combinatorial examples of Cohen-Macaulay rings.

More recently, the Cohen-Macaulay notion has been extended to non-Noetherian rings (see \cite{glaz1992coherence}, \cite{glaz1995homological}, \cite{hamilton2007non}, and \cite{asgharzadeh2009notion}, the latter giving a nice survey). While there are many equivalent definitions of a Cohen-Macaulay ring in the Noetherian setting, it turns out that these definitions may no longer be the same for non-Noetherian rings. This gives rise to many different notions of Cohen-Macaulayness in the non-Noetherian setting, the strongest among them being \emph{Cohen-Macaulay in the sense of ideals} and \emph{Weak Bourbaki unmixed} (see \cite{asgharzadeh2009notion} for an explanation of the connections between the various definitions). One method of obtaining non-Noetherian Cohen-Macaulay rings was given by Asgharzadeh, Dorreh, and Tousi who showed in \cite[Corollary 3.8]{asgharzadeh2014direct} that any flat direct limit of Noetherian Cohen-Macaulay rings is both Cohen-Macaulay in the sense of ideals and weak Bourbaki unmixed. We call such rings \emph{Cohen-Macaulay in the sense of flat direct limits}.

In \cite{chlopecki2026antidiagonal}, Chlopecki, Meintjes, and the first author proved that if $k$ is an algebraically closed field, $\Delta$ is a (potentially infinite-dimensional) simplicial complex, and $\Delta_1 \subseteq \Delta_2 \subseteq \ldots \subseteq \Delta$ is an increasing sequence of finite full Cohen-Macaulay subcomplexes whose union is $\Delta$, then $k[\Delta]$ is a flat direct limit of the Stanley-Reisner rings $k[\Delta_n]$, and is therefore Cohen-Macaulay in the sense of flat direct limits. 

In this paper, we mimic Bj\"{o}rner and Kind and Kleinschmidt's  ``lexicographically shellable implies shellable implies Cohen-Macaulay'' pipeline in the setting of an infinite poset. We introduce the notion of \emph{lexicographically nested-shellablility} for an infinite poset and \emph{nested-shellability} for an infinite simplicial complex, and we prove that as in the finite case, the former implies the latter, which implies via \cite[Theorem A]{chlopecki2026antidiagonal} Cohen-Macaulayness in the sense of flat direct limits. This is recorded in the following theorems. (We fix $k$ to be an algebraically closed field throughout the rest of the paper.)   

\begin{lettertheorem}\label{thm: main thm nested implies CM}
   If $\Delta$ is a nested-shellable simplicial complex, its Stanley-Reisner ring $k[\Delta]$ is Cohen-Macaulay in the sense of flat directed limits. 
\end{lettertheorem}

\begin{lettertheorem}\label{thm: main thm nested-lex implies nested}
    If $P$ is a lexicographically nested-shellable poset, then the order complex $\Delta(P)$ is nested-shellable. 
\end{lettertheorem}

We then apply these results to give an analog of Edelman's theorem (\cite{edelman1981bruhat}) for the infinite symmetric group $S_\infty$ of all bijections $\nn \to \nn$. More specifically, in \cite{gallup2021well} the first author introduced a version of the Bruhat order for $S_\infty$. Here we prove the following theorem that certain intervals $[\sigma , \omega]^f \subseteq S_\infty$ consisting of permutations between $\sigma$ and $\omega$ in Bruhat order which are of finite length above $\sigma$ are lexicographically nested-shellable, generating new examples of rings which are Cohen-Macaulay in the sense of ideals and weak Bourbaki unmixed. 

\begin{lettertheorem}\label{thm: main thm Bruhat order}
If $\sigma , \omega \in S_\infty$ and $\sigma <_\text{Bruhat} \omega$, then the poset $([\sigma , \omega]^f , <_\text{Bruhat})$ is lexicographically nested-shellable, and hence the Stanley-Reisner ring $k[\Delta([\sigma , \omega]^f)]$ of its order complex $\Delta([\sigma , \omega]^f)$ is Cohen-Macaulay in the sense of flat direct limits. 
\end{lettertheorem}

The layout of the paper is as follows. In Section \ref{sec: Nested-Shellable Simplicial Complexes} we define nested-shellable simplicial complexes and show that they are Cohen-Macaulay in the sense of flat direct limits. In Section \ref{sec: Lexicographically Nested-Shellable Posets} we extend the notion of a grading to infinite posets, use this ``$\nn$-grading'' to define lexicographically nested-shellable posets, and show that the latter definition implies nested-shellable. In the remainder of the paper we apply these results to the intervals $[\sigma , \omega]^f$ in $S_\infty$. While a surprising amount of the combinatorics of the Bruhat order on $S_n$ remains true in $S_\infty$, the proof methods are often different, since in the infinite setting we do not yet have access to many of the tools of Coxeter groups. Hence we develop much of this combinatorial theory from scratch along the way. Specifically, in Section \ref{sec: Bruhat Order on the Infinite Symmetric Group} we recall the definition of the infinite Bruhat order and give two equivalent characterizations of it. It turns out that like in $S_n$, all cover relations in $S_\infty$ are given by composing with transpositions, and in Section \ref{sec: Cover Relations in the Bruhat Order} we describe exactly when $\sigma \circ (p , q)$ covers $\sigma$ in the infinite Bruhat order. We also give a similar description of exactly when $\sigma <_\text{Bruhat} \sigma \circ (p , q) \leq_\text{Bruhat} \omega$ in the relative situation that $\sigma <_\text{Bruhat} \omega$. In Section \ref{sec: discrete-saturated Chains in the Bruhat Order} we use this relative description of cover relations to show that for any $\sigma <_\text{Bruhat} \omega$ there exists a discrete-saturated chain (whose cover relations are given by transpositions of a particularly nice form) beginning at $\sigma$ and either stopping at $\omega$ or converging to $\omega$ in a natural topology. As a result we obtain a full characterization of cover relations in $S_\infty$. In Section \ref{sec: An N Grading of S infty} we show that $[\sigma , \omega]^f$ is $\nn$-graded with rank function a relative version of Coxeter length. Finally, in Section \ref{sec: Lexicographically Nested-Shellable Intervals of S infty} we show that the intervals $[\sigma , \omega]^f$ are lexicographically nested-shellable, obtaining our proof of Theorem \ref{thm: main thm Bruhat order}.  


\section{Nested-Shellable Simplicial Complexes}\label{sec: Nested-Shellable Simplicial Complexes}

In this short section we review the notion of a (not necessarily finite) simplicial complex, define nested-shellable simplicial complexes, and show they are Cohen-Macaulay in the sense of flat direct limits. 

A \emph{simplicial complex} on a countable set $V$ (called the \emph{vertex set}) is a collection $\Delta$ of \textbf{finite} subsets of $V$ (called \emph{faces}) with the property that if $A \in \Delta$ and $B \subseteq A$ then $B \in \Delta$. $\Delta$ is a \emph{finite simplicial complex} if it is a finite set. A face $F$ of $\Delta$ is called a \emph{facet} if it is maximal with respect to inclusion among the set of faces. A finite-dimensional simplicial complex $\Delta$ is called \emph{pure} if all of its facets have the same size. Note that if $V$ is finite, then every face is contained in a facet, but if $\Delta$ is infinite, this is no longer true. Indeed, it may be that $\Delta$ has no facets. 

Given a finite subset $A \subseteq V$, define $\textbf{x}^A = \prod_{v \in A} x_v \in k[x_v \mid v \in V]$. Note that this is a square-free monomial. If $\Delta$ is a simplicial complex with vertex set $V$, define $I_{\Delta, V}$ to be the ideal of $k[x_v \mid v \in V]$ generated by the set of square-free monomials $\{ \textbf{x}^A \mid A \subseteq V, A \text{ is finite}, A \notin \Delta \}$. When $V$ is finite, Reisner showed thar the map $\Delta \mapsto I_{\Delta, V}$ is a bijection between the set of simplicial complexes on $V$ and the set of square-free monomial ideals in $k[x_v \mid v \in V ]$ (see \cite{reisner1976cohen}). In fact because we have required the faces of an infinite simplicial complex to be finite, it is still true that even when $V$ is infinite, this map is a bijection (see \cite[Proposition 2]{chlopecki2026antidiagonal} and \cite[Remark 1]{chlopecki2026antidiagonal}). For more background on (not necessarily finite) simplicial complexes and positively graded Cohen-Macaulay $k$-algebras, see \cite{chlopecki2026antidiagonal}. 

Recall (see \cite[5.1.11]{bruns1998cohen}) that a finite pure simplicial complex $\Delta$ is called \emph{shellable} if the facets of $\Delta$ can be given a linear order $F_1, \ldots, F_m$ so that for all $2 \leq i \leq m$, we have that $\left\langle F_i\right\rangle \cap\left\langle F_1, \ldots, F_{i-1}\right\rangle$ is generated by a non-empty set of maximal proper faces of $\left\langle F_i\right\rangle$. Kind and Kleinschmidt showed that the Stanley-Reisner ring $k[\Delta]$ of a finite pure shellable simplicial complex $\Delta$ is Cohen-Macaulay (\cite{kind_schalbare_1979}). In the following definition we give an analog of shellability for infinite simplicial complexes, which we then show implies Cohen-Macaulay in the sense of flat direct limits (Theorem \ref{thm: main thm nested implies CM}).

\begin{definition}
    We say that a simplicial complex $\Delta$ is \emph{nested-shellable} if there exists an increasing sequence $\Delta_1 \subseteq \Delta_2 \subseteq \ldots$ of finite full subcomplexes of $\Delta$ such that (NS1) each $\Delta_i$ is shellable and such that (NS2) $\Delta = \bigcup_{n \in \nn} \Delta_n$. 
\end{definition}

\begin{proof}[Proof of Theorem \ref{thm: main thm nested implies CM}]
    By Kind and Kleinschmidt's result (\cite{kind_schalbare_1979}), each face ring $k[\Delta_n]$ is Cohen-Macaulay, hence the desired result follows from \cite[Theorem A]{chlopecki2026antidiagonal}. 
\end{proof}


\section{Lexicographically Nested-Shellable Posets}\label{sec: Lexicographically Nested-Shellable Posets}

We begin this section by reviewing some necessary background on posets and gradings of finite posets. We then extend the notion of grading to infinite posets, define the lexicographically nested-shellable condition, and show that it implies the nested-shellable condition of the previous section. 

Given elements $p , q \in P$, the \emph{interval from $p$ to $q$} is defined to be the set $[p , q]_P = \{ p' \in P \mid p \leq p' \leq q\}$. We will drop the subscript ``$P$'' if it is clear from context. For $p , q \in P$, recall that $q$ is said to \emph{cover} $p$ if $p < q$ and there does not exist $p' \in P$ such that $p < p' < q$, i.e. if $[p , q]$ consists of exactly two elements. We denote by $C(P) = \{ ( p , q) \mid p \lessdot q \}$ the set of cover relations of $P$. A \emph{chain} in a poset $P$ is a subset $C \subseteq P$ which is totally ordered, and a chain is \emph{maximal} if it is not properly contained in another chain. Furthermore, a chain $C$ in $P$ is called \emph{saturated} if it has the property that whenever $p , q \in C$ and $p' \in P$ is such that $p < p' < q$, then $p' \in C$ as well. A finite or infinite chain in $P$ is called \emph{discrete-saturated} if it is saturated and of the form $p_0 < p_1 < \ldots$, i.e. is finite or is order-isomorphic to $\nn$. Note that a chain of this form is discrete-saturated if and only if $p_n \lessdot p_{n + 1}$ for all $n$. In this case we write $C: p_0 \lessdot p_1 \lessdot \ldots$ to denote that $C$ is the discrete-saturated chain consisting of the elements $p_i$. 

\begin{remark}
    In a finite poset, maximal chains are always saturated but of course the converse fails. However in a finite poset a saturated chain is maximal if and only if it contains a minimal and maximal element. The situation is more complicated for infinite posets. For example $\qq$ as a subset of itself with the usual order is both a maximal and saturated chain, but it contains neither a minimal nor maximal element, and it is not discrete-saturated because it has the wrong order-type. In this paper we shall only be concerned with discrete-saturated chains, and we discuss some partial results about them in Lemma \ref{lem: discrete-saturated chains are maximal in graded posets}. 
\end{remark} 

Following \cite{bjorner1980shellable}, an \emph{edge labeling} of a poset $P$ with values in a poset $Q$ is a function $\lambda: C(P) \to Q$. The \emph{Jordan-H\"{o}lder sequence} of a finite or infinite discrete-saturated chain $C: p_0 \lessdot p_1 \lessdot \ldots$ in $P$ with respect to $\lambda$ is the ordered list $J_\lambda(C) = (\lambda(p_0, p_1) , \lambda(p_1, p_2) , \ldots )$ with entries in $Q$. A discrete-saturated chain $p_0 \lessdot p_1 \lessdot \ldots$ is $\lambda$-\emph{increasing} if $\lambda(p_0 , p_1) < \lambda(p_1, p_2) < \ldots$ is an increasing chain in $Q$. 

A \emph{graded poset} is one which is finite, bounded (has unique minimal and maximal elements), and pure (every maximal chain has the same size and hence the poset admits a rank function). Graded posets have the Jordan-Dedekind property that for any pair of elements $p \leq p'$, every saturated chain starting at $p$ and ending at $p'$ has the same length. 

If $\lambda$ is an edge-labeling of a graded poset $P$ then it is called an \emph{$L$-labeling} if (L1) for all $p < q$ there exists a unique $\lambda$-increasing maximal chain $C$ in $[p , q]$, and (L2) the Jordan-H\"{o}lder sequence of $C$ lexicographically precedes that of any other maximal chain in $[p , q]$. (Note that to compare two Jordan-H\"{o}lder sequences in the lexicographical ordering it is necessary that they have the same length, which is guaranteed by the graded condition.) A graded poset is called \emph{lexicographically shellable} if it has an $L$-labeling. 

The \emph{order complex} of a poset $P$ is the simplicial complex $\Delta(P)$ with vertex set $P$ whose faces are the finite chains in $P$. Note that the squarefree monomial ideal $I_{\Delta(P) , P} \subseteq k[x_p \mid p \in P]$ corresponding to $\Delta(P)$ under the Stanley-Reisner bijection is generated by the monomials of the form $x_p x_q$ where $p$ and $q$ are not related in $P$. 

Bj\"{o}rner proved that the order complexes of lexicographically shellable posets are shellable (\cite{bjorner1980shellable}) and we shall prove a similar statement for infinite posets, but we first need analogs of these definitions that work in the infinite setting. We begin by defining a notion of grading for infinite posets that will guarantee the Jordan-Dedekind property for all finite intervals. 

\begin{definition}\label{def: graded}
    A poset $P$ is called $\nn$-\emph{graded} if the following conditions are satisfied. 
\begin{itemize}
    \item[(G1)] $P$ has a unique minimal element (which we will denote by $\hat{0}$).
    \item[(G2)] For every $p \in P$, there exists a finite saturated chain starting at $\hat{0}$ and ending at $p$. 
    \item[(G3)] For every $p \in P$, every maximal chain in $[\hat{0} , p]$ has the same length (which must be finite by (G2)).  
\end{itemize}
We will denote this length by $\rho(p)$ and call $\rho: P \to \zz_{\geq 0}$ the \emph{rank function} of $P$.
\end{definition} 

\begin{remark}
In Definition \ref{def: graded} we have not required $P$ to have a maximal element since our main posets of interest $[\sigma , \omega]^f$ do not have maximal elements, as we show in Corollary \ref{cor: no max element for intervals either}. In fact it is also true that entire poset $S_\infty$ does not have a maximal element (Corollary \ref{cor: no maximal elements}). 
\end{remark}

\begin{remark}
    If $P$ is finite and $\nn$-graded in the sense of Definition \ref{def: graded}, then $P$ may not be graded in the traditional sense of Bj\"{o}rner since it might not have a maximal element, and still may not be graded even upon adding in such an element since there may be maximal elements of $P$ that have different ranks. However it is always the union of finitely many graded posets, with length functions that agree with $\rho$ where they are mutually defined.
\end{remark} 

We actually want guarantee that all intervals in our infinite posets are graded, but this property necessitates requiring said finiteness directly.

\begin{definition} 
We say that a poset $P$ is \emph{interval-finite} if for any $p , p' \in P$ with $p \leq p'$, the interval $[p , p'] = \{ q \in P \mid p \leq q \leq p' \}$ is finite.
\end{definition}

It turns out that an $\nn$-graded poset need not be interval-finite as the following example shows.  

\begin{example}
   Let $P = \{ (0 , 0), (0 , 2) , (n , 1) \mid n \in \zz \} \subseteq \zz^2$ be a poset with the product order (i.e. $(p , q) < (p' , q')$ if and only if $p < p'$ and $q < q'$). Then $P$ is clearly an $\nn$-graded poset with rank function $\rho(p , q) = q$, but is not interval-finite, since $P = [(0 , 0) , (0 , 2)]$ which is infinite. 
\end{example}

However it is true that an $\nn$-graded poset which is interval-finite has the desired property that any interval is graded. 

\begin{lemma}\label{lem: N-graded plus interval-finite means locally graded}
    If $P$ is an $\nn$-graded poset which is interval-finite, then for any $p, q \in P$ with $p < q$ we have that $[p , q]$ is graded.
\end{lemma}

\begin{proof}
By the interval-finite hypothesis, $[p , q]$ is a finite poset, which clearly has a unique minimal element $p$ and a unique maximal element $q$. By (G2) there exists a finite saturated chain $C$ starting at $\hat{0}$ and ending at $p$. Concatenating any two maximal chains $D$ and $D'$ of $[p , q]$ with $C$ yield two maximal chains of $[\hat{0} , q]$ which must have the same length by (G3). Hence $D$ and $D'$ must also have the same length, so $[p , q]$ is pure.
\end{proof}

We can now give our analog of the lexicogrpahically shellable condition for infinite posets and then prove Theorem \ref{thm: main thm nested-lex implies nested} that lexicographically nested-shellable implies nested-shellable, which is the infinite analog of Bj\"{o}rner's aforementioned theorem.  

\begin{definition}
    We say that a poset $P$ is \emph{lexicographically nested-shellable} if it is interval-finite, $\nn$-graded, and the following conditions are satisfied. 
    
    \begin{enumerate}
        \item[(LNS1)] There exists an edge-labeling $\lambda$ of $P$ such that for each $p \leq q$, $\lambda$ is an $L$-labeling for the poset $[p , q]$ (which is necessarily graded by Lemma \ref{lem: N-graded plus interval-finite means locally graded}). 
        \item[(LNS2)] There exists a chain $\hat{0} < p_1 < \ldots$ such that for all $p \in P$ there exists some $n \in \nn$ such that $p \leq p_n$. 
    \end{enumerate} 
\end{definition}

\begin{proof}[Proof of Theorem \ref{thm: main thm nested-lex implies nested}]
    First of all $P$ is interval-finite, so $[\hat{0} , p_n]$ is finite for all $n$. Next note that if $p , q \in P$ then $\Delta([p , q])$ is a full subcomplex of $\Delta(P)$. Indeed if $q_1 < \ldots < q_n$ is a chain in $[p , q]$ then certainly it is a chain in $P$, and if $q_1 < \ldots < q_n$ is a chain in $P$ such that $q_i \in [p , q]$ for each $i$ then it is a chain in $[p , q]$ as well.

    Therefore we have an increasing sequence $\Delta([\hat{0} , p_1]) \subseteq \Delta([\hat{0} , p_2]) \subseteq \ldots$ of finite full subcomplexes of $\Delta(P)$. By Lemma \ref{lem: N-graded plus interval-finite means locally graded}, for each $i \in \nn$, $[\hat{0} , p_i]$ is a graded poset which is lexicographically shellable by (LNS1), hence by \cite[Theorem 2.3]{bjorner1980shellable}, $\Delta([\hat{0} , p_i])$ is shellable so (NS1) is satisfied. Finally by (LNS2), given any $p \in P$, there exists some $n \in \nn$ such that $p \leq p_n$, hence $\{p\} \in \Delta([\hat{0} , p_n])$. Therefore given a face $F$ of $\Delta(P)$, there exists some $n \in \nn$ such that all vertices of $F$ are vertices in $\Delta([\hat{0} , p_n])$. Because $\Delta([\hat{0} , p_n])$ is a full subcomplex of $\Delta(P)$, it follows that $F \in \Delta([\hat{0} , p_n])$, so indeed we have $\bigcup_{n \in \nn} \Delta([\hat{0} , p_n]) = \Delta(P)$, and (NS2) is satisfied as well. 
\end{proof}

The remainder of the paper is dedicated to applying these definitions and results to the poset $S_\infty$ with the Bruhat order. 


\section{Bruhat Order on the Infinite Symmetric Group}\label{sec: Bruhat Order on the Infinite Symmetric Group}

In this section we define the Bruhat order on the infinite symmetric group and give two equivalent formulations of it that will be useful. We denote by $S_\infty$ the group of all bijections $\mathbb{N} \to \mathbb{N}$ (the group law being composition, the identity denoted by $e$) and by $S_\infty^f$ the subgroup of those bijections which fix all but finitely many elements of $\mathbb{N}$. There is an inclusion of $S_n$ into $S_{n + 1}$ by fixing $n + 1$, so we obtain a nested chain of subgroups $S_1 \subseteq S_2 \subseteq S_3 \subseteq \ldots \subseteq S_\infty$ whose union is $S_\infty^f$. Note that $S_\infty$ contains many elements that are not in $S_\infty^f$. Indeed the latter is countable while the former is not. Given $\omega \in S_\infty$, the \emph{one line notation} of $\omega$ is the ordered list $[\omega(1), \omega(2), \ldots]$. 

\begin{example}\label{ex: first perm}
    The permutation $\rho \in S_\infty$ which sends all odd $n$ to $n + 2$, all even $n$ with $n \neq 2$ to $n - 2$, and $2$ to $1$ has one-line notation $[3, 1, 5, 2, 7, 4, 9, 6, 11 , \ldots ]$. Note that $\rho$ is not an element of $S_\infty^f$. 
\end{example}

A \emph{transposition} in $S_\infty$ is a permutation which swaps two natural numbers and leaves all others fixed. We usually use the \emph{cycle notation} $(p , q)$ to denote the transposition which sends $p \mapsto q$ and $q \mapsto p$ and fixes all other numbers. Throughout this paper whenever we write a transposition in the cycle notation $(p , q)$ we will always assume that $p < q$.

In \cite[Section 9]{gallup2021well}, the second author described a topology on $S_\infty$ in which a sequence $\sigma_1, \sigma_2, \ldots$ of permutations converges to a permutation $\sigma$ if and only if for all $m \in \nn$ there exists some $N \in \nn$ such that if $n \geq N$ then $\sigma_n(m) = \sigma(m)$. In other words this sequence converges if and only if for all $m \in \nn$ the sequence $\sigma_1(m), \sigma_2(m), \ldots$ is eventually equal to $\sigma(m)$. 

\begin{example}\label{ex: two infinite permutations}
    Consider the permutation $\theta$ which maps all odd $n$ to $n + 1$ and all even $n$ to $n - 1$. The one-line notation of $\theta$ is $[2,1,4,3,6,5,8,7, \ldots]$, and again notice that $\theta \notin S_\infty^f$. Note that the sequence $(1 , 2), (1 , 2) \circ (3 , 4) , (1 , 2) \circ (3, 4) \circ (5 , 6) , \ldots$ of permutations in $S_\infty^f$ converges to $\theta$, as illustrated below in one-line notation. 
    \begin{align*}
    e &=[1,2,3,4,5,6,7,8 \ldots]
    \\ (1,2) &= [2,1,3,4,5,6,7,8 \ldots]
    \\(1,2)\circ (3,4)&= [2,1,4,3,5,6,7,8 \ldots]
    \\(1,2)\circ (3,4) \circ (5,6) &=[2,1,4,3,6,5,7,8 \ldots ]
    \\& \hspace{18.5 mm} \vdots
    \\& \hspace{17 mm} \downarrow
    \\\hspace{1 mm} \theta & = [2,1,4,3,6,5,8,7 \ldots]
    \end{align*}
\end{example}

The Bruhat order is a partial order on $S_n$ which is a combinatorial description of the containment relations among Schubert varieties in the flag variety. It was first studied in this geometric context by Ehresmann (\cite{ehresmann1934topologie}) and Chevalley (\cite{chevalley1994decompositions}), and then in the context of algebraic combinatorics by Verma (\cite{verma1966structure}). The Bruhat order has several equivalent characterizations (see \cite{bjorner2005combinatorics} for details as well as a development of the relevant combinatorics). 

In \cite{gallup2021well}, the second author generalized the Bruhat order to one on $S_\infty$ which we call the \emph{infinite Bruhat order} and now describe. Suppose that $X$ and $Y$ are subsets of $\nn$ both of cardinality $n$, and that the elements of $X$ are denoted by $x_i$ for $1 \leq i \leq n$ such that $x_1 < x_2 < \ldots < x_n$ and similarly for $Y$. Then we say that $X \leq Y$ if and only if $x_i \leq y_i$ for all $i$. This defines a partial order on the set of subsets of $\nn$ with cardinality $n$. Using this, we define the \emph{infinite Bruhat order} on $S_\infty$, denoted by $\leq_\text{Bruhat}$, by setting $\sigma \leq_\text{Bruhat} \omega$ if and only if for every $n \in \nn$, we have that $\{ \sigma(1) , \ldots, \sigma(n) \} \leq \{ \omega(1), \ldots, \omega(n) \}$. 

\begin{remark}
    If we restrict the infinite Bruhat order to $S_n$ (which we identify with the subgroup of $S_\infty$ which fixes all $m > n$) we obtain the classical Bruhat order.
\end{remark} 

The infinite Bruhat order has an equivalent formulation in terms of Young tableaux. To a permutation $\sigma \in S_\infty$, we associate an infinite tableau of shape $(1 , 2 , 3 , \ldots )$ where the $n$th row has entries $\sigma(1) , \sigma(2) , \ldots, \sigma(n)$ written in increasing order. 
Then for another permutation $\omega$ we have that $\sigma \leq_\text{Bruhat} \omega$ if and only if the entry in each box of the tableau associated to $\sigma$ is less than or equal to the entry in the corresponding box in the tableau associated to $\omega$.

\begin{example}
Consider the permutations $\theta = [2,1,4,3,6,5,8,7, \ldots]$ and  $\rho = [3, 1, 5, 2, 7, 4, 9, 6, 11 , \ldots ]$ defined in Examples \ref{ex: first perm} and \ref{ex: two infinite permutations}. Using the tableau criterion it is not difficult to show that $\theta <_\text{Bruhat} \rho$, as the following diagram indicates. 
\begin{align*}
\begin{matrix}
    \young(2,12,124,1234,12346,123456,1234568) & <_\text{Bruhat} &\young(3,13,135,1235,12357,123457,1234579)
\\ \vdots & &  \vdots
\end{matrix}
\end{align*}
\end{example}


The infinite Bruhat order also has a third description in terms of permutation matrices. To any permutation $\sigma \in S_\infty$ we can assign a permutation matrix $[\sigma_{m ,n}]_{m , n \in \nn}$ where $\sigma_{m , n}$ is $1$ if $m = \sigma(n)$ and zero otherwise. For any ordered pair of subsets $X , Y \subseteq \nn$ we denote by $r_{X , Y}(\sigma)$ the number of $1$'s in the submatrix $[\sigma_{m , n}]_{m \in X, n \in Y}$ (which is its rank). Note that equivalently we could have written: 
\begin{equation*}
    r_{X , Y}(\sigma) = \# \{ p \in Y \mid \sigma(p) \in X \}.
\end{equation*} 

We adopt the following notational shorthand: for any $a , b \in \nn$ define $r_{a , b}(\sigma) = r_{[1, a],[1 , b]}(\sigma)$ and $r_{<a , b}(\sigma) = r_{[1, a),[1 , b]}(\sigma)$ (as well as other reasonable combinations of these notations). Thus $r_{a , b}(\sigma)$ is the number of $1$'s in the upper left submatrix of the permutation matrix of $\sigma$ whose lower right corner is in row $a$ and column $b$. We can use $r_{a , b}(\sigma)$ to characterize the Bruhat order on $S_\infty$. 

\begin{proposition}\label{prop: region criterion for bruhat order}
Given $\sigma, \omega \in S_\infty$, the following are equivalent. 

\begin{enumerate}
    \item $\sigma \leq_\text{Bruhat} \omega$
    \item $r_{a , b}(\sigma) \geq r_{a , b}(\omega)$ for all $a , b \in \nn$
    \item $r_{\geq a , b}(\sigma) \leq r_{\geq a , b}(\omega)$ for all $a , b \in \nn$
\end{enumerate}
\end{proposition}

\begin{proof}
The equivalence of (1) and (2) was proved in \cite[Proposition 27]{gallup2021well}, and (2) is equivalent to (3) because $r_{\geq a + 1 , b}(\sigma) = b - r_{a , b}(\sigma)$. 
\end{proof}

\begin{remark}\label{rmk: rs equal implies permutations are equal}
    Proposition \ref{prop: region criterion for bruhat order}, along with the fact that the Bruhat order is actually a partial order on $S_\infty$, implies that if $r_{a , b}(\sigma) = r_{a , b}(\omega)$ for all $a , b \in \nn$ then $\sigma = \omega$. 
\end{remark}


\section{Cover Relations in the Bruhat Order}\label{sec: Cover Relations in the Bruhat Order}

In this section we give a criterion for when $\sigma \circ (p , q)$ covers $\sigma$ in $S_\infty$ preceded some technical definitions and lemmas. We then give a similar criterion for when $\sigma \lessdot_\text{Bruhat} \sigma \circ (p , q) \leq_\text{Bruhat} \omega$ in the relative situation when $\sigma <_\text{Bruhat} \omega$. These criteria are of course well known for $S_n$ (see, e.g. \cite[Lemma 2.1.4]{bjorner2005combinatorics}), but our proof methods differ from those typically used in $S_n$ because we do not currently have the tools of reduced words, inversion number, and length in $S_\infty$, so we give a precise development to be thorough. 

First note that if we write $\omega \in S_\infty$ in one-line notation as $\omega = [a_1, a_2, a_3, \ldots a_p, \ldots, a_q, \ldots]$ then $\omega \circ (p,q) = [a_1, a_2, a_3, \ldots a_q, \ldots, a_p, \ldots]$, i.e. composing $\omega$ with the transposition $(p,q)$ on the right swaps elements $a_p$ and $a_q$ in the one-line notation while leaving all other elements fixed. We now prove a lemma which describes the difference between the permutation matrices of $\sigma$ and $\sigma \circ (p , q)$ for certain $p$ and $q$. 

\begin{lemma}\label{lem: ones in boxes after applying a transposition}
    Suppose $\sigma \in S_\infty$ and $p , q \in \nn$ are such that $p < q$ and $\sigma(p) < \sigma(q)$. Then we have the following formulas. 
    \begin{equation*}
        r_{a , b}(\sigma) = \left\{ \begin{matrix} r_{a , b}(\sigma \circ (p , q)) & \text{ if } & b < p, b \geq q , a < \sigma(p), \text{ or } a \geq \sigma(q) \\ r_{a , b}(\sigma \circ (p , q)) + 1 & \text{ if } & \sigma(p) \leq a < \sigma(q) \text{ and } p \leq b < q  \end{matrix} \right. 
    \end{equation*} 
\end{lemma}

\begin{proof}
Since the matrix of $\sigma \circ (p , q)$ is obtained from that of $\sigma$ by swapping columns $p$ and $q$, notice that if $b < p$,  $b \geq q$, $a < \sigma(p)$, or $a \geq \sigma(q)$ then the submatrix $[(\sigma \circ (p , q))_{m , n}]_{1 \leq m \leq a , 1 \leq n \leq b}$ is either identical to or obtained by rearranging the $1$'s from the submatrix $[\sigma_{m , n}]_{1 \leq m \leq a , 1 \leq n \leq b}$, hence $r_{a, b}( \sigma \circ (p , q) ) = r_{a, b}(\sigma)$. 

On the other hand, if $\sigma(p) \leq a < \sigma(q)$ and $p \leq b < q$, then the submatrix $[\sigma_{m , n}]_{1 \leq m \leq a , 1 \leq n \leq b}$ is obtained from $[(\sigma \circ (p , q))_{m , n}]_{1 \leq m \leq a , 1 \leq n \leq b}$ by adding a $1$ in the $(\sigma(p) , p)$ position, so we have $r_{a, b}( \sigma ) = r_{a, b}(\sigma \circ (p , q)) + 1$.
\end{proof}

Next we further compare $\sigma$ and $\sigma \circ (p , q)$ using the following definition. 

\begin{definition}\label{def: d}
    For any $\sigma , \omega \in S_\infty$ with $\sigma \neq \omega$, we denote by $d(\sigma , \omega)$ the first $n \in \nn$ such that $\sigma(n) \neq \omega(n)$.
\end{definition}

We record the following lemma, which was proved in \cite[Lemma 25]{gallup2021well}.

\begin{lemma}\label{lem: less than bruhat implies values at d less than}
    If $\sigma , \omega \in S_\infty$ are such that $\sigma <_\text{Bruhat} \omega$, then $\sigma(d(\sigma , \omega)) < \omega(d(\sigma , \omega))$. 
\end{lemma}

Thus if $\sigma <_\text{Bruhat} \omega$ then since $\sigma(\ell) = \omega(\ell)$ for all $\ell < d(\sigma , \omega)$, it must be that $d(\sigma , \omega) < f(\sigma , \omega)$ and $\sigma(d(\sigma, \omega)) < \sigma(f(\sigma , \omega))$ where we define $f(\sigma , \omega) := \sigma^{-1}(\omega(d(\sigma , \omega)))$. Hence there exists a minimal $m$ with $d(\sigma , \omega) < m \leq f(\sigma, \omega)$ such that $\sigma(d( \sigma , \omega)) < \sigma(m) \leq \sigma(f(\sigma , \omega))$. We denote this minimal $m$ by $m(\sigma , \omega)$. Minimality here means that if $d(\sigma , \omega) < \ell < m(\sigma , \omega)$ then either $\sigma(\ell) < \sigma(d(\sigma , \omega))$ or $\sigma(\ell) > \sigma(f(\sigma , \omega))$.

\begin{example}
Consider the permutations $\theta = [2,1,4,3,6,5,8,7, 10, \ldots]$ and $\rho = [3, 1, 5, 2, 7, 4, 9, 6, 11 , \ldots ]$ defined in Examples \ref{ex: first perm} and \ref{ex: two infinite permutations} and recall that $\theta <_\text{Bruhat} \rho$. Then we have $d(\theta,\rho) = 1$ and $m(\theta,\rho) = f(\theta,\rho) = 4$. 
\end{example}

Now we can prove our criterion for when composing on the right by $(p , q)$ goes up in the Bruhat order. 

\begin{proposition}\label{prop: transposition going up in bruhat}
    For any $\sigma \in S_\infty$, and any $p < q$, we have that $\sigma <_\text{Bruhat} \sigma \circ (p , q)$ if and only if the following condition is satisfied.
    \begin{enumerate}
        \item[(C1)] $\sigma(p) < \sigma(q)$
    \end{enumerate} 
\end{proposition}

\begin{proof}
Suppose first that $\sigma <_\text{Bruhat} \sigma \circ (p , q)$. Then $p = d(\sigma , \sigma \circ (p, q))$ and so by Lemma \ref{lem: less than bruhat implies values at d less than} we have that $\sigma(p) <[\sigma \circ (p , q)] (p) = \sigma(q)$ as desired. Now suppose that $\sigma(p) < \sigma(q)$. By Lemma \ref{lem: ones in boxes after applying a transposition}, for any $a , b \in \nn$, we have $r_{a, b}( \sigma \circ (p , q) ) \leq r_{a, b}(\sigma)$, which implies by Proposition \ref{prop: region criterion for bruhat order} that $\sigma \leq_\text{Bruhat} \sigma \circ (p , q)$. Since $\sigma \neq \sigma \circ (p , q)$, the result follows. 
\end{proof}

From this proposition we obtain the following corollary which is intuitively obvious but we include for completeness. 

\begin{corollary}\label{cor: no maximal elements}
The poset $(S_\infty , <_\text{Bruhat})$ does not have a maximal element.
\end{corollary}

\begin{proof}
If $\sigma \in S_\infty$, then given any $p \in \nn$, there must exist some $q > p$ such that $\sigma(p) < \sigma(q)$. By Proposition \ref{prop: transposition going up in bruhat} we have that $\sigma <_\text{Bruhat} \sigma \circ (p ,  q)$, so $\sigma$ cannot be maximal. 
\end{proof}

So far we have determined when $\sigma \circ (p , q)$ is larger than $\sigma$ in the infinite Bruhat order, but now we will characterize exactly when this is a cover, preceded by another technical lemma that will be needed for the proof. 

\begin{lemma}\label{lem: criteria for equal permutation matrices}
    Given $a , b \in \nn$, and $\sigma , \omega \in S_\infty$, we have that $\sigma_{a , b} = \omega_{a , b}$ if the following conditions are satisfied.  
    \begin{enumerate}
        \item $r_{<a , b}(\sigma) = r_{<a , b}(\omega)$
        \item $r_{a , <b}(\sigma) = r_{a , <b}(\omega)$
        \item $r_{<a , <b}(\sigma) = r_{<a , <b}(\omega)$
        \item $r_{a , b}(\sigma) = r_{a , b}(\omega)$
    \end{enumerate} 
\end{lemma}

\begin{proof}
Suppose $\sigma_{a , b} = 1$. Then $\sigma_{i , b} = 0$ for all $i < a$, and hence $r_{< a , < b}(\sigma) = r_{< a , b}(\sigma)$. But by hypothesis we have $r_{< a , < b}(\sigma) = r_{< a , < b}(\omega)$ and also that $r_{< a , b}(\omega) = r_{< a , b}(\sigma)$, which implies that $r_{< a , < b}(\omega) = r_{< a , b}(\omega)$ as well, therefore $\omega_{i , b} = 0$ for all $i < a$ as well, which in particular implies that $r_{<a , [b , b]}(\omega) = 0$. However $r_{a , < b}(\omega) + r_{a , [b , b]}(\omega) = r_{a , b}(\omega) 
= r_{a , b}(\sigma) = r_{a , < b}(\sigma) + r_{a , [b , b]}(\sigma)$. Since $r_{a , < b}(\omega) = r_{a , < b}(\sigma)$ by hypothesis, subtraction yields $r_{a , [b , b]}(\omega) = r_{a , [b , b]}(\sigma) = 1$. Since $r_{<a , [b , b]}(\omega) = 0$ and $r_{a , [b , b]}(\omega) = r_{<a , [b , b]}(\omega) + \omega_{a , b}$, again subtraction yields $\omega_{a , b} = 1$. 

Since the situation is entirely symmetric in $\sigma$ and $\tau$, it follows that $\sigma_{a , b} = 1$ if and only if $\omega_{a , b} = 1$, and since the only possible values for $\sigma_{a , b}$ and $\omega_{a , b}$ are $0$ or $1$, the result follows.
\end{proof}

\begin{proposition}\label{prop: cover relations in the bruhat order}
    Let $\sigma \in S_\infty$ and $p, q \in \nn$ be such that $p < q$. Then $\sigma \lessdot_\text{Bruhat} \sigma \circ (p , q)$ if and only if the following conditions are satisfied. 
    \begin{enumerate}
        \item[(C1)] $\sigma(p) < \sigma(q)$
        \item[(C2)] For all $p < \ell < q$ we have that either $\sigma(\ell) < \sigma(p)$ or $\sigma(\ell) > \sigma(q)$.
    \end{enumerate}
\end{proposition}

\begin{proof}
First of all, suppose $\sigma \lessdot_\text{Bruhat} \sigma \circ (p , q)$. By Proposition \ref{prop: transposition going up in bruhat} we have $\sigma(p) < \sigma(q)$, so (C1) is satisfied. 

To show (C2), suppose that $p < \ell < q$ and that $\sigma(p) < \sigma(\ell) < \sigma(q)$. We claim that $\sigma <_\text{Bruhat} \sigma \circ (p , \ell) <_\text{Bruhat} \sigma \circ (p , q)$. The first inequality holds by Proposition \ref{prop: transposition going up in bruhat}. On the other hand, we have that $\sigma \circ (p , q) = \sigma \circ (p , \ell) \circ (\ell , q) \circ (p , \ell)$. Note that $[\sigma \circ (p , \ell)](\ell) = \sigma(p) < \sigma(q) = [\sigma \circ (p , \ell)](q)$, therefore by Proposition \ref{prop: transposition going up in bruhat} we have $\sigma \circ (p , \ell) <_\text{Bruhat} \sigma \circ (p , \ell) \circ (\ell , q)$. Similarly we have that $[\sigma \circ (p , \ell) \circ (\ell , q)](p) = \sigma(\ell) < \sigma(q) = [\sigma \circ (p , \ell) \circ (\ell , q)](\ell)$, so again by Proposition \ref{prop: transposition going up in bruhat} we have $\sigma \circ (p , \ell) \circ (\ell , q) <_\text{Bruhat} \sigma \circ (p , \ell) \circ (\ell , q) \circ (p , \ell) = \sigma \circ (p , q)$. So indeed we have $\sigma <_\text{Bruhat} \sigma \circ (p , \ell) <_\text{Bruhat} \sigma \circ (p , q)$ contradicting that $\sigma \lessdot_\text{Bruhat} \sigma \circ (p , q)$. Therefore (C2) is satisfied as well.

Now suppose $(p , q)$ satisfies (C1) and (C2). By Proposition \ref{prop: transposition going up in bruhat} we have $\sigma <_\text{Bruhat} \sigma \circ (p , q)$. Suppose that for some $\tau \in S_\infty$ we have $\sigma \leq_\text{Bruhat} \tau \leq_\text{Bruhat} \sigma \circ (p , q)$. By Proposition \ref{prop: region criterion for bruhat order}, for all $a , b \in \nn$ we have $r_{a , b}(\sigma) \geq r_{a , b}(\tau) \geq r_{a , b}(\sigma \circ (p , q))$. However by Lemma \ref{lem: ones in boxes after applying a transposition}, if $(\ast)$ $b < p$, $b \geq q$, $a < \sigma(p)$, or $a \geq \sigma(q)$ then $r_{a , b}(\sigma) = r_{ a , b}(\sigma \circ (p , q))$, implying that $r_{a , b}(\sigma) = r_{a , b}(\tau) = r_{ a , b}(\sigma \circ (p , q))$, and if $(\ast \ast)$ $\sigma(p) \leq a < \sigma(q)$ and $p \leq b < q$, then $r_{a , b}(\sigma) = r_{a , b}(\sigma \circ (p , q)) + 1$, so either $r_{a , b}(\tau) = r_{a , b}(\sigma)$ or $r_{a , b}(\tau) = r_{ a , b}(\sigma \circ (p , q))$.

First, we claim that $(\ast \ast \ast)$ $\sigma_{i , j} = \tau_{i , j}$ in the following (not disjoint) regions: (A) $i \in \nn$ and $j < p$, (B) $i < \sigma
(p)$ and $j \in \nn$, (C) $\sigma(q) < i$ and $j \in \nn$, (D) $i \in \nn$ and $q < j$. In any of the cases, since $\nn \times \nn$ is well-founded in the product order, there is a minimal $(i , j)$ in the relevant region such that $\sigma_{i , j} \neq \tau_{i , j}$. But $(\ast)$ implies that with $a = i$ and $b  = j$ the hypotheses of Lemma \ref{lem: criteria for equal permutation matrices} are satisfied, hence in fact $\sigma_{i , j} = \tau_{i , j}$, contradiction. 

Now we claim that $\sigma(p) \leq \tau(p) \leq \sigma(q)$. Indeed, we compute:  
\begin{equation*}
    r_{<\sigma(p) , < p}(\tau) \overset{(1)}{=} r_{<\sigma(p) , < p}(\sigma) \overset{(2)}{=} r_{<\sigma(p) , p}(\sigma) \overset{(3)}{=} r_{<\sigma(p) , p}(\tau)
\end{equation*}
Here (1) and (3) follow from $(\ast)$ and (2) follows because for all $1 \leq i < \sigma(p)$ we have $\sigma_{i , p} = 0$.  This implies that $\tau_{i , p} = 0$ for all $1 \leq i < \sigma(p)$, i.e. $\sigma(p) \leq \tau(p)$. Similarly we compute: 
\begin{equation*}
    r_{\sigma(q) , <p}(\tau) + 1 \overset{(1)}{=} r_{\sigma(q) , <p}(\sigma) + 1 \overset{(2)}{=} r_{\sigma(q) , p}(\sigma) \overset{(3)}{=} r_{\sigma(q) , p}(\tau)
\end{equation*}
Here again (1) and (3) follow from $(\ast)$ and (2) follows because $\sigma(p) < \sigma(q)$ by (C1) and $\sigma_{i , p} = 0$ unless $i = \sigma(p)$ in which case $\sigma_{i , p} = 1$. Therefore $\tau_{i , p} = 1$ for some $1 \leq i \leq \sigma(q)$, i.e. $\tau(p) \leq \sigma(q)$. Putting these two inequalities together yields the desired claim.  

Next we claim that in fact $\tau(p) \in \{ \sigma(p), \sigma(q) \}$. Suppose, for contradiction, that $\sigma(p) < \tau(p) < \sigma(q)$ and let $j = \sigma^{-1}(\tau(p))$, i.e. $\sigma_{\tau(p) , j} = 1$. Then we cannot have $j < p$, since by $(\ast \ast \ast)$ region (A), we have $\sigma_{i , j} = \tau_{i , j}$ for all $i \in \nn$ and $j < p$ . We also cannot have $j = p$, since $\sigma(p) < \tau(p)$ by hypothesis. Furthermore (C2), together with the fact that $\sigma(p) < \tau(p) < \sigma(q)$, implies that we cannot have $p < j < q$ either. Finally we cannot have $j = q$ since $\sigma(q) > \tau(p)$. Thus it must be that $j > q$. But then by $(\ast \ast \ast)$ region (D) we have that $\tau_{\tau(p) , j} = \sigma_{\tau(p) , j} = 1$, i.e. $\tau(j) = \tau(p)$, contradicting that $\tau$ is a bijection, so the claim holds.

Finally, we claim that $r_{[\sigma(p) , \sigma(q)],[p , q)}(\tau) = 1$. We first break up the region $[1 , \sigma(q)] \times [1 , q)$ into three disjoint reigons $[\sigma(p) , \sigma(q)] \times [1 , p)$, $[1 , \sigma(p)) \times [1 , q)$, and  $[\sigma(p) , \sigma(q)] \times [p , q)$, yielding the following. 
\begin{align*}
    r_{[\sigma(p) , \sigma(q)] , [1 , p)}(\tau) + r_{[1 , \sigma(p)) , [1 , q)}(\tau) + r_{[\sigma(p) , \sigma(q)],[p , q)}(\tau) 
    &= r_{[1 , \sigma(q)],[1 , q)}(\tau) 
    \\&\overset{(1)}{=} r_{[1 , \sigma(q)],[1 , q)}(\sigma) 
    \\&= r_{[\sigma(p) , \sigma(q)] , [1 , p)}(\sigma) + r_{[1 , \sigma(p)) , [1 , q)}(\sigma) + r_{[\sigma(p) , \sigma(q)],[p , q)}(\sigma)
\end{align*}
Here (1) follows from $(\ast)$. Since $\sigma(j) = \tau(j)$ for all $j < p$ we have $r_{[\sigma(p) , \sigma(q)] , [1 , p)}(\sigma) = r_{[\sigma(p) , \sigma(q)] , [1 , p)}(\tau)$. Also we have $r_{[1 , \sigma(p)) , [1 , q)}(\sigma) = r_{[1 , \sigma(p)) , [1 , q)}(\tau)$ again by $(\ast)$. Therefore subtraction yields $r_{[\sigma(p) , \sigma(q)],[p , q)}(\tau) = r_{[\sigma(p) , \sigma(q)],[p , q)}(\sigma)$. By (C2), we know that $r_{[\sigma(p) , \sigma(q)],[p , q)}(\sigma) = 1$, so the claim follows. 

Since $\tau(p) \in \{ \sigma(p) , \sigma(q) \}$, the previous claim implies that $\tau_{i , j} = 0$ for all $p \leq j < q$ and $\sigma(p) \leq i \leq \sigma(q)$ unless $j = p$ and either $i = \sigma(p)$ or $i = \sigma(q)$. 

By $(\ast \ast \ast)$ also we have that $\tau_{i , j} = \sigma_{i , j}$ in regions (A), (B), (C), and (D), so the only $(i , j)$ for which $\tau_{i , j}$ and $\sigma_{i , j}$ can differ is $(i , j) \in \{ (\sigma(p) , p), (\sigma(q), p) , (a , q) \mid \sigma(p) \leq a \leq \sigma(q) \}$, and exactly one of $\tau_{\sigma(p) , p}$ and $\tau_{\sigma(q) , p}$ must be $1$. Therefore $\tau(j) = \sigma(j)$ for all $j \neq p , q$, and so because $\tau$ is a bijection it must be that in the former case, $\tau(q) = \sigma(q)$ and in the latter case $\tau(q) = \sigma(p)$. Thus in the former case $\tau = \sigma$ and in the latter $\tau = \sigma \circ (p , q)$, as desired.
\end{proof}

In the remainder of this section, we give a version of Proposition \ref{prop: cover relations in the bruhat order} that is relative to a given permutation $\omega$. Specifically, given $\sigma <_\text{Bruhat} \omega$, we characterize for which transpositions $(p , q)$ we have that $\sigma \lessdot_\text{Bruhat} \sigma \circ (p , q) \leq_{\text{Bruhat}} \omega$.  We give a name to such a transposition below. 

\begin{definition}
    Let $\sigma , \omega \in S_\infty$ and suppose $\sigma <_\text{Bruhat} \omega$. We say that a transposition $(p , q)$ is a \emph{relative candidate} for the pair $\sigma , \omega$ if 
    \begin{equation*}\label{eq: relative candidate equation}
        \sigma \lessdot_\text{Bruhat} \sigma \circ (p , q) \leq_\text{Bruhat} \omega.
    \end{equation*} 
\end{definition}

\begin{proposition}\label{prop: equivalent characterization of relative candidate}
     Let $\sigma , \omega \in S_\infty$ and suppose $\sigma <_\text{Bruhat} \omega$. Then $(p , q)$ is a relative candidate for $\sigma, \omega$ if and only if the following conditions are satisfied.  
     \begin{enumerate}
        \item[(C1)] $\sigma(p) < \sigma(q)$
        \item[(C2)] For all $p < \ell < q$ we have that either $\sigma(\ell) < \sigma(p)$ or $\sigma(\ell) > \sigma(q)$.
        \item[(C3)] For all $a , b \in \nn$ we have $r_{a , b}(\sigma) \geq r_{a , b}(\omega)$ with the inequality strict if $p \leq b < q$ and $\sigma(p) \leq a < \sigma(q)$. 
    \end{enumerate}
\end{proposition}

\begin{proof}
First of all, suppose $(p , q)$ is a relative candidate for $\sigma , \omega$. By definition we have that $\sigma \lessdot_\text{Bruhat} \sigma \circ (p , q)$, so Proposition \ref{prop: cover relations in the bruhat order} implies that (C1) and (C2) are satisfied. 

To show (C3) we note first that since $\sigma <_\text{Bruhat} \omega$, by Proposition \ref{prop: region criterion for bruhat order} we have $r_{a , b}(\sigma) \geq r_{a , b}(\omega)$ for any $a , b \in \nn$. However if $p \leq b < q$ and $\sigma(p) \leq a < \sigma(q)$, then by Lemma \ref{lem: ones in boxes after applying a transposition}, we have $r_{a , b}(\sigma) = r_{a , b}(\sigma \circ (p , q)) + 1$ and since $\sigma \circ (p , q) \leq_\text{Bruhat} \omega$ it follows that $r_{a , b}(\sigma \circ (p , q)) \geq r_{a , b}(\omega)$, therefore $r_{a , b}(\sigma) > r_{a , b}(\omega)$ as desired. 

Now suppose $(p , q)$ satisfies (C1), (C2), and (C3). Again Proposition \ref{prop: cover relations in the bruhat order} implies $\sigma \lessdot_{\text{Bruhat}} \sigma \circ (p , q)$. Furthermore, Lemma \ref{lem: ones in boxes after applying a transposition} implies that if $b < p$, $b \geq q$, $a < \sigma(p)$, or $a \geq \sigma(q)$, then $r_{a , b}(\sigma \circ (p , q)) = r_{a , b}(\sigma)$ and (C3) implies that $r_{a , b}(\sigma) \geq r_{ a , b}(\omega)$, hence $r_{a , b}(\sigma \circ (p , q)) \geq r_{ a , b}(\omega)$. On the other hand, if $p \leq b < q$ and $\sigma(p) \leq a < \sigma(q)$, then again Lemma \ref{lem: ones in boxes after applying a transposition} implies $r_{a , b}(\sigma) = r_{a , b}(\sigma \circ (p , q)) + 1$, but by (C3) we have $r_{a , b}(\sigma) > r_{a , b}(\omega)$, hence $r_{a , b}(\sigma \circ (p , q)) \geq r_{a , b}(\omega)$. Therefore in both cases we have $r_{a , b}(\sigma \circ (p , q)) \geq r_{a , b}(\omega)$, so Proposition \ref{prop: region criterion for bruhat order} implies $\sigma \circ (p , q) \leq_\text{Bruhat} \omega$, and hence $(p , q)$ is a relative candidate for $\sigma, \omega$. 
\end{proof}


\section{Saturated Chains in the Bruhat Order}\label{sec: discrete-saturated Chains in the Bruhat Order}

The goal of this section is to prove that for any pair of permutations $\sigma , \omega \in S_\infty$ with $\sigma <_\text{Bruhat} \omega$, there exists a discrete-saturated chain in the Bruhat order $\sigma = \sigma_0 \lessdot_\text{Bruhat} \sigma_1 \lessdot_\text{Bruhat} \ldots \leq_\text{Bruhat} \omega$ which is either eventually equal to $\omega$ or converges to $\omega$ in the aforementioned topology on $S_\infty$. In fact we show that this sequence can be obtained by repeatedly composing with transpositions of the form $(d , m)$ (see Definition \ref{def: d} and the discussion below), which will be important to us later. A similar result was proved in \cite{gallup2021well}, but the directions of the inequalities there are reversed, so for completeness we prove it here in the direction that will be useful to us. We begin with a key lemma which shows there always exists at least one relative candidate for any pair $\sigma <_\text{Bruhat} \omega$. 

\begin{lemma}\label{lem: d m is a relative candidate}
     Suppose $\sigma, \omega \in S_\infty$ are such that $\sigma <_\text{Bruhat} \omega$. Then $(d(\sigma, \omega) , m(\sigma , \omega))$ is a relative candidate for $\sigma , \omega$.
\end{lemma}

\begin{proof}
Let $d = d(\sigma , \omega)$, $m = m(\sigma , \omega)$, and $f = f(\sigma , \omega)$. First of all, (C1) and (C2) follow trivially from the definitions of $d$ and $m$. By hypothesis $\sigma <_\text{Bruhat} \omega$, hence by Proposition \ref{prop: region criterion for bruhat order} we have $r_{a , b}(\sigma) \geq r_{a , b}(\omega)$ for all $a , b \in \nn$. Now suppose that $d \leq b < m$ and $\sigma(d) \leq a < \sigma(m)$. We break up the matrix region $[1 , \omega(d) ] \times [1 , b]$ into the three disjoint regions $[1 , a] \times [1  , b]$, $(a , \omega(d)] \times [1 , d)$, and $(a , \omega(d)] \times [d , b]$ and compute as follows. 
\begin{equation*}
    r_{a , b}(\sigma) + r_{(a , \omega(d)], [1 , d)}(\sigma) + r_{(a , \omega(d)], [d , b]}(\sigma) = r_{\omega(d) , b}(\sigma) \overset{(1)}{\geq} r_{\omega(d) , b}(\omega) = r_{a , b}(\omega) + r_{(a , \omega(d)], [1 , d)}(\omega) + r_{(a , \omega(d)], [d , b]}(\omega)
\end{equation*}
Here (1) follows because $\sigma <_\text{Bruhat} \omega$. However, by definition of $m$, the submatrix $[\sigma_{i , j}]_{\sigma(d) \leq i \leq \omega(d) , d \leq j < m}$ has exactly one $1$ in the $\sigma(d)$th row and the $d$th column and $0$'s elsewhere, therefore $r_{(a , \omega(d)], [d , b]}(\sigma) = 0$. Furthermore, the permutation matrix $[\omega_{i , j}]$ has a $1$ in the $\omega(d)$th row and the $d$th column and therefore $r_{(a , \omega(d)], [d , b]}(\omega) > 0$. Finally, since $\sigma(p) = \omega(p)$ for all $p < d$, we have $r_{(a , \omega(d)], [1 , d]}(\sigma) = r_{(a , \omega(d)], [1 , d]}(\omega)$. Therefore subtraction yields $r_{a , b}(\sigma) \geq r_{a , b}(\omega) + r_{(a , \omega(d)], [d , b]}(\omega) > r_{a , b}(\omega)$, so (C3) is satisfied.
\end{proof}

We now use this lemma to prove the existence of the promised discrete-saturated chains in $S_\infty$ whose cover relations are given by transpositions of the form $(d , m)$. 

\begin{proposition}\label{prop: existence of a discrete-saturated chain}
    Let $\sigma , \omega \in S_\infty$ be such that $\sigma <_\text{Bruhat} \omega$. Let $\sigma_0 = \sigma$ and for $i \geq 1$ inductively define $\sigma_i = \sigma_{i - 1} \circ (d(\sigma_{i  - 1} , \omega) , m(\sigma_{i - 1} , \omega))$ as long as $\sigma_{i - 1} \neq \omega$. If $\sigma_n = \omega$ for some $n$, then $\sigma  = \sigma_0 \lessdot_\text{Bruhat} \sigma_1 \lessdot_\text{Bruhat} \ldots \lessdot_\text{Bruhat} \sigma_{n - 1} \lessdot_\text{Bruhat} \sigma_n = \omega$ is a finite saturated chain. If $\sigma_n \neq \omega$ for all $n \in \nn$, then $\sigma  = \sigma_0 \lessdot_\text{Bruhat} \sigma_1 \lessdot_\text{Bruhat} \ldots \leq_\text{Bruhat} \omega$ is a discrete-saturated chain which converges to $\omega$.
\end{proposition}

\begin{proof}
First of all, according to Lemma \ref{lem: d m is a relative candidate} for all $i$ we have that $(d(\sigma_{i - 1}, \omega), m(\sigma_{i - 1},\omega))$ is a relative candidate for $\sigma_{i - 1},\omega$, so by definition we have $\sigma_{i - 1} \lessdot_\text{Bruhat} \sigma_{i - 1} \circ (d(\sigma_{i - 1}, \omega), m(\sigma_{i - 1},\omega)) = \sigma_i \leq_\text{Bruhat} \omega$. Thus in either case the chain obtained is discrete-saturated, and if $\sigma_n = \omega$ for some $n$, the claim is proved.  

Now suppose that $\sigma_n \neq \omega$ for all $n$. We now show that the sequence $\sigma = \sigma_0, \sigma_1, \sigma_2, \ldots$ converges to $\omega$. In fact we claim that $d(\sigma_0 , \omega)$, $d(\sigma_1 , \omega)$, $d(\sigma_2 , \omega), \ldots$ is a weakly increasing sequence which is not eventually constant. To see that it is weakly increasing, note that for any $i$ we have $\sigma_{i + 1}(\ell) = \sigma_i(\ell) = \omega(\ell)$ for all $\ell < d(\sigma_i , \omega)$, hence $d(\sigma_i , \omega) \leq d(\sigma_{i + 1} , \omega)$. Suppose, for contradiction, that this sequence is eventually constant. First note that if $d(\sigma_{n + 1} , \omega) = d(\sigma_{n}, \omega)$ then $f(\sigma_{n + 1} , \omega) = f(\sigma_{n} , \omega)$. By assumption we have $d(\sigma_\ell , \omega) =  d(\sigma_{n} , \omega)$ for all $\ell \geq n$, hence $f(\sigma_\ell , \omega) = f(\sigma_n , \omega)$ for all $\ell \geq n$. But by the properties of $d$, $m$, and $f$ we have that $d(\sigma_\ell , \omega) < m(\sigma_\ell , \omega) \leq f(\sigma_\ell , \omega)$ for all $\ell \geq n$, so $d(\sigma_n , \omega) < m(\sigma_\ell , \omega) \leq f(\sigma_n, \omega)$ for all $\ell \geq n$. However there are only finitely many integers $m$ which satisfy $d(\sigma_n , \omega) < m \leq f(\sigma_n, \omega)$, hence there exists some $n' \geq n$ such that for all $\ell \geq n'$ we have $m(\sigma_\ell , \omega) = m(\sigma_{n'} , \omega)$. In particular we have $d(\sigma_{n'} , \omega) = d(\sigma_{n' + 1} , \omega)$ and $m(\sigma_{n'} , \omega) = m(\sigma_{n' + 1} , \omega)$. However since $\sigma_{n'} (d(\sigma_{n'} , \omega) ) < \sigma_{n'} (m(\sigma_{n'} , \omega) )$ and $\sigma_{n' + 1} = \sigma_{n'} \circ (d(\sigma_{n'} , \omega) , m(\sigma_{n'} , \omega) )$, we obtain the following. 
\begin{equation*}
    \sigma_{n' + 1}( d(\sigma_{n' + 1} , \omega) ) = \sigma_{n'}( m(\sigma_{n'} , \omega) ) > \sigma_{n'}( d(\sigma_{n'} , \omega) ) = \sigma_{n' + 1}( m(\sigma_{n' + 1} , \omega) )
\end{equation*}
This contradicts the definition of $m(\sigma_{n' + 1} , \omega)$. Therefore, given any $i \in \nn$, there exists some $N \in \nn$ such that for all $n \geq N$ we have $d(\sigma_n , \omega) > i$, which implies $\sigma_n(i) = \omega(i)$ for $n \geq N$. Hence $\sigma = \sigma_0 , \sigma_1, \ldots$ indeed converges to $\omega$ as desired. 
\end{proof}

As a result, we obtain the following complete characterization of cover relations in $S_\infty$. 

\begin{corollary}\label{cor: cover relations in S infinity are given by transpositions}
    Given $\sigma , \omega \in S_\infty$, we have that $\sigma \lessdot_\text{Bruhat} \omega$ if and only if there exists a transposition $(p , q)$ which satisfies (C1), (C2), and $\omega = \sigma \circ (p , q)$. In this case the transposition $(p , q)$ is unique. 
\end{corollary}

\begin{proof}
    If $\sigma \lessdot_\text{Bruhat} \omega$, Proposition \ref{prop: existence of a discrete-saturated chain} implies in particular that there exists a discrete-saturated chain $\sigma \lessdot_\text{Bruhat} \sigma_1 \lessdot_\text{Bruhat} \ldots \leq_\text{Bruhat} \omega$ with each cover relation given by a transposition (which must satisfy (C1) and (C2) by Proposition \ref{prop: cover relations in the bruhat order}). Since $\sigma \lessdot_\text{Bruhat} \omega$ is already a cover relation, it must be that $\sigma_1 = \omega$, so the claim is proven. Proposition \ref{prop: cover relations in the bruhat order} immediately gives the reverse implication. 
    To see uniqueness, note that $(p , q) = \sigma^{-1} \circ \omega$. 
\end{proof}


\section{An $\nn$-Grading of Intervals in $S_\infty$}\label{sec: An N Grading of S infty}

This section is dedicated to defining the subsets $[\sigma , \omega]^f$ of $S_\infty$ and showing that they are $\nn$-graded. Recall that the \emph{length} of a permutation $\sigma \in S_\infty^f$ is the number $\ell(\sigma)$ of inversions of $\sigma$, i.e. the size of the set $\{(i , j) \in \nn \times \nn \mid i < j \text{ and } \sigma(i) > \sigma(j)\}$. The poset $S_n$ with the Bruhat order is graded with rank equal to length (see, e.g., \cite{bjorner2005combinatorics}). 

Given $\sigma, \nu \in S_\infty$ we say that $\sigma$ and $\nu$ are \emph{eventually equal} if $\sigma(n) \neq \nu(n)$ for only finitely many $n \in \nn$, and in this case we write $\sigma \approx \nu$. Note that ``$\approx$'' is an equivalence relation on $S_\infty$. 

\begin{definition}
    Given any $\sigma <_\text{Bruhat} \omega$, we define 
\begin{equation*}
    [\sigma , \omega]^f = \{ \nu \in S_\infty \mid \sigma \leq_\text{Bruhat} \nu \leq_\text{Bruhat} \omega \text{ and } \nu \approx \sigma \}.
\end{equation*}
\end{definition}

We wish to show that $[\sigma , \omega]^f$ is $\nn$-graded, but we need a notion of length of a permutation that makes sense in this relative setting to give us a rank function. Of course if $\sigma \in S_\infty \smallsetminus S^f_\infty$ then $\sigma$ will have infinitely many inversions, so the classical length of most elements in $S_\infty$ is infinite. However the following analog of length does make sense for all permutations.

\begin{definition}
The \emph{pseudo-length} of $\sigma \in S_\infty$ is the vector $\ell_\bullet(\sigma) = (\ell_1(\sigma), \ell_2(\sigma), \ldots) \in \nn^{\zz_{\geq 0}}$ defined by setting $\ell_i(\sigma) = \sum_{n = 1}^i |D_n(\sigma)|$, where we define $D_i(\sigma)$ to be the set $\{ j \in \nn \mid i < j, \sigma(i) > \sigma(j) \}$ of \emph{$i$th inversions of $\sigma$}. In other words $\ell_{i}(\sigma)$ is the number of $n$th inversions for $n \leq i$. 
\end{definition}

Notice that $\ell_i(\sigma)$ is always a finite number because there are only finitely many elements of $\nn$ smaller than $\sigma(i)$ for any $i \in \nn$.

\begin{example}
Consider the permutation $\theta = [2 , 1 , 4 , 3 , 6 , 5 , \ldots]$ and $\rho = [3 , 1 , 5 , 2 , 7 , 4, 9 , 6 , \ldots ]$ from Examples \ref{ex: first perm} and \ref{ex: two infinite permutations}. We have that $|D_i(\theta)| = 1$ if $i$ is odd and $0$ if $i$ is even, and hence $\ell_\bullet(\theta) = (1, 1 , 2 , 2 , 3 , 3 , \ldots)$. Similarly $|D_i(\rho)| = 2$ if $i$ is odd and $0$ if $i$ is even, so $\ell_\bullet(\rho) = (2, 2 , 4 , 4 , 6 , 6 , \ldots)$. 
\end{example}

Note also that if $\sigma \in S_\infty^f$, then $|D_i(\sigma)| = 0$ for all but finitely many $i$, and hence $\ell_i(\sigma) = \ell(\sigma)$ for all but finitely many $i$. In fact this implication can be reversed.

\begin{proposition}\label{prop: eventually constant ell implies in f}
    Given $\sigma \in S_\infty$, we have that $\sigma \in S_\infty^f$ if and only if there exists $n \in \nn$ such that $\ell_i(\sigma) = n$ for all but finitely many $i$. In this case, $\ell(\sigma) = n$. 
\end{proposition}

\begin{proof}
As already remarked, the forward implication is obvious. For the reverse implication, by hypothesis, there exists some $i_0$ such that $\ell_{i}(\sigma) = n$ for all $i \geq i_0$. This implies that $D_i(\sigma) = \emptyset$ for all $i \geq i_0$, which by definition means that $\sigma(i) > \sigma(i_0)$ for all $i > i_0$. But this implies that if $\sigma(j) < \sigma(i_0)$, then $j < i_0$, so there are $\sigma(i_0) - 1$ numbers which must appear to the left of $i_0$ in the one-line notation of $\sigma$, hence $i_0 \geq \sigma(i_0)$.

For any $n \in \nn$, we have $\sigma(i_0) < \ldots < \sigma(i_0 + (n - 1)) < \sigma(i_0 + n)$ and if $\sigma(i_0 + n) > \sigma(i_0 + (n - 1)) + 1$ then since $\sigma(i_0 + (n - 1)) + 1 > \sigma(i_0 + (n - 1)) > \ldots > \sigma(i_0)$, $\sigma(i_0 + (n - 1)) + 1$ must appear before the $i_0$ position in the one-line notation of $\sigma$. Since there are only finitely many such positions before $i_0$, there must exist some $i_1 \geq i_0$ such that $\sigma(i_1 + n) = \sigma(i_1) + n$ for all $n \in \nn$, i.e. such that the one-line notation of $\sigma$ consists of consecutive numbers after $\sigma(i_1)$. We claim that in fact $\sigma(i_1) = i_1$. Indeed, since the numbers after $\sigma(i_1)$ in one-line notation are consecutive, the numbers $1 , \ldots, \sigma(i_1) - 1$ must appear before $\sigma(i_1)$, and these are the only possible numbers that can appear before $\sigma(i_1)$, so the claim follows. Since $\sigma(i_1 + n) = \sigma(i_1) + n$ for all $n \in \nn$, it therefore follows that $\sigma(i) = i$ for all $i \geq i_1$ as well, so $\sigma \in S_\infty^f$. 
\end{proof}

We now use pseudo-length to give a notion of length in $S_\infty$ relative to a given permutation. Note that for any $\sigma, \nu \in S_\infty$ that are eventually equal, there exists some $i_0 \in \nn$ such that $D_i(\sigma) = D_i(\nu)$ for all $i \geq i_0$. Therefore the difference $\ell_i(\nu) - \ell_i(\sigma)$ stabilizes as $i \to \infty$, i.e. there exists some $n \in \zz_{\geq 0}$ such that $\ell_i(\nu) - \ell_i(\sigma) = n$ for all $i \geq i_0$. 

\begin{definition}
    If $\sigma , \nu \in S_\infty$ are such that $\sigma <_\text{Bruhat} \nu$ and $\sigma \approx \nu$, the \emph{length of $\nu$ relative to $\sigma$} is defined to be $\ell_\sigma(\nu) = \lim_{ i \to \infty} \ell_i(\nu) - \ell_i(\sigma)$. 
\end{definition}

Note that if $\sigma, \nu \in S_\infty^f$, then there is some $i_0 \in \nn$ such that if $i \geq i_0$ then $\ell_i(\sigma) = \ell(\sigma)$ and $\ell_i(\nu) = \ell(\nu)$, and therefore $\ell_i(\nu) - \ell_i(\sigma) = \ell(\nu) - \ell(\sigma)$, hence relative length in $S_\infty^f$ is just given by $\ell_\sigma(\nu) = \ell(\nu) - \ell(\sigma)$, as one would expect. 

\begin{remark}
Note that two elements $\sigma, \nu \in S_\infty$ are eventually equal if and only if $\sigma^{-1} \circ \nu \in S^f_\infty$. Indeed, we have that $\sigma \approx \nu$ if and only if there exists some $N \in \nn$ such that if $n \geq N$, $\sigma(n) = \nu(n)$, which is true if and only if for all $n \geq N$ we have $(\sigma^{-1} \circ \nu)(n) = n$, which is true if and only if $\sigma^{-1} \circ \nu \in S_\infty^f$. 
Since the usual length of a permutation is defined for all elements of $S_\infty^f$, one might be tempted to define $\ell_\sigma(\omega) = \ell(\sigma^{-1} \circ \omega)$, however this is not even true for permutations in $S_n$. For example $[1,3,2] \lessdot_\text{Bruhat} [2,3,1]$ in $S_3$ so the relative length is $\ell_{[1,3,2]}([2,3,1]) = \ell([2,3,1]) - \ell([1,3,2]) = 1$, but we have $[1,3,2] \circ (1,3) = [2 , 3 , 1]$, hence $\ell([1,3,2]^{-1} \circ [2,3,1]) = \ell( (1 , 3) ) = 3$.  
\end{remark}

We now show that relative length is additive.

\begin{lemma}\label{lem: relative length is additive}
    If $\sigma , \mu , \nu \in S_\infty$ are all eventually equal, and satisfy $\sigma \leq_\text{Bruhat} \mu \leq_\text{Bruhat} \nu$, then $\ell_\sigma(\nu) = \ell_\sigma(\mu) + \ell_\mu(\nu)$.
\end{lemma}

\begin{proof}
    By hypothesis there exists some $i_0 \in \nn$ such that if $i \geq i_0$ then $\sigma(i) = \mu(i) = \nu(i)$. Hence we have $\ell_\sigma(\mu) = \ell_{i_0}(\mu) - \ell_{i_0}(\sigma)$, $\ell_\mu(\nu) = \ell_{i_0}(\nu) - \ell_{i_0}(\mu)$, and $\ell_\sigma(\nu) = \ell_{i_0}(\nu) - \ell_{i_0}(\sigma)$. Adding the first two equations together yields $\ell_\sigma(\nu) = \ell_\sigma(\mu) + \ell_\mu(\nu)$. 
\end{proof}

We can use this to show that like in $S_n$, cover relations in the Bruhat order on $S_\infty$ increase relative length by $1$.

\begin{lemma}\label{lem: cover relations increase relative length by one}
    Given $\sigma, \nu \in S_\infty$, if $\sigma \lessdot_\text{Bruhat} \nu$, then $\ell_\sigma(\nu) = 1$.  
\end{lemma}

\begin{proof}
    By Corollary \ref{cor: cover relations in S infinity are given by transpositions} there exists some transposition $(p , q)$ satisfying (C1) and (C2) such that $\nu = \sigma \circ (p , q)$. Let $i_0 = q + 1$ and note that $\nu(i) = \sigma(i)$ for all $i \geq i_0$, therefore $\ell_\sigma(\nu) = \ell_{i_0}(\nu) - \ell_{i_0}(\sigma)$. 

    Suppose $i < j$ with $i \leq i_0$. We consider several cases. (\emph{Case 1}). If $i , j \notin \{ p , q \}$, then clearly $(i , j)$ is an inversion of $\nu$ if and only if it is an inversion of $\sigma$. (\emph{Case 2}). If $i < p$ and $j \in \{ p , q \}$, then $(i , p)$ is an inversion of $\sigma$ if and only if $(i , q)$ is an inversion of $\nu$ and $(i , p)$ is an inversion of $\nu$ if and only if $(i , q)$ is an inversion of $\sigma$. (\emph{Case 3}). Suppose $i = p$ and $j < q$. If $(p , j)$ is an inversion of $\sigma$, then since $\sigma(p) < \sigma(q)$, $(p , j)$ is an inversion of $\nu$. Conversely, if $(p , j)$ is an inversion of $\nu$, then $\sigma(q) > \sigma(j)$, so by (C2) we must also have $\sigma(p) > \sigma(j)$, hence $(p , j)$ is an inversion of $\sigma$ too. (\emph{Case 4}). If $i = p$ and $q < j$, then $(p , j)$ is an inversion of $\sigma$ if and only if $(q , j)$ is an inversion of $\nu$. (\emph{Case 5}). Suppose $p < i$ and $j = q$. If $( i , q)$ is an inversion of $\sigma$, then $\sigma(i) < \sigma(q)$, so by (C2), $\sigma(i) < \sigma(p)$, hence $(i , q)$ is an inversion of $\nu$ as well. Conversely, if $(i , q)$ is an inversion of $\nu$ then $\sigma(i) < \sigma(p) < \sigma(q)$, so $(i , q)$ is an inversion of $\sigma$. (\emph{Case 6}). Finally if $i = q$, then $(q , j)$ is an inversion of $\sigma$ if and only if $(p , j)$ is an inversion of $\nu$. Then, since $(p , q)$ is an inversion of $\nu$ but not of $\sigma$, the result follows.
\end{proof}

We now describe the intervals $[\sigma , \omega]^f$ when $\sigma \approx \omega$. 

\begin{proposition}\label{prop: ee implies intervals are equal}
    Given any $\sigma , \omega \in S_\infty$ with $\sigma <_\text{Bruhat} \omega$ and $\sigma \approx \omega$, we have that $[\sigma , \omega]^f = [\sigma , \omega]$.
\end{proposition}

\begin{proof}
    The inclusion ``$\subseteq$'' holds by definition. For the reverse inclusion, let $\nu \in [\sigma , \omega]$. By hypothesis, $\sigma$ and $\omega$ are eventually equal, thus there exists some $i_0 \in \nn$ such that $\sigma(i) = \omega(i)$ for all $i \geq i_0$. This implies that the sets $\{ \sigma(1) , \ldots, \sigma(i_0  - 1) \}$ and $\{ \omega(1) , \ldots, \omega(i_0  - 1) \}$ must be equal. Since $\sigma <_\text{Bruhat} \nu <_\text{Bruhat} \omega$, by definition $\{ \sigma(1) , \ldots, \sigma(i_0  - 1) \} \leq \{ \nu(1) , \ldots, \nu(i_0  - 1) \} \leq \{ \omega(1) , \ldots, \omega(i_0  - 1) \}$. But since the first and last sets are equal, the middle set must be equal to the first (and last) as well. The same result is also clearly true if $i_0$ is replaced by $i_0 + n$ for any $n \in \nn$. Hence by induction on $n$, it must be that $\sigma(i_0 + n) = \nu(i_0 + n) = \omega(i_0 + n)$ for all $n \in \nn$. But this implies that $\sigma(i) = \nu(i) = \omega(i)$ for all $i \geq i_0$. Hence $\nu \approx \sigma$, so $\nu \in [\sigma , \omega]^f$. 
\end{proof}

We immediately obtain that for any $\mu , \nu \in [\sigma, \omega]^f$, the interval $[\mu , \nu]_{[\sigma, \omega]^f}$ of elements between them in the poset $[\sigma, \omega]^f$ is actually equal to the full interval $[\mu , \nu] = [\mu , \nu]_{S_\infty}$ between them in the poset $S_\infty$. For clarity we keep the subscript on ``$[\mu , \nu]_{S_\infty}$'' throughout the corollary.  

\begin{corollary}\label{cor: big interval contained in ee interval}
    Given any $\sigma , \omega \in S_\infty$ with $\sigma <_\text{Bruhat} \omega$, and any $\mu , \nu \in [\sigma , \omega]^f$ with $\mu <_\text{Bruhat} \nu$, we have that $[\mu , \nu]_{S_\infty} = [\mu , \nu]_{[\sigma , \omega]^f}$. 
\end{corollary}

\begin{proof}
    The inclusion ``$\supseteq$'' is obvious. For the reverse inclusion, let $\tau \in [\mu , \nu]_{S_\infty}$. By definition both $\mu$ and $\nu$ are eventually equal to $\sigma$, and hence are eventually equal to each other. Therefore, by Proposition \ref{prop: ee implies intervals are equal}, $[\mu , \nu]_{S_\infty} = [\mu , \nu]^f$, so $\tau \in [\mu , \nu]^f$ which implies $\tau$ is eventually equal to $\mu$ and therefore to $\omega$. Since $\sigma \leq_\text{Bruhat} \mu \leq_\text{Bruhat} \tau \leq_\text{Bruhat} \nu \leq_\text{Bruhat} \omega$ it follows that $\tau \in [\mu , \nu]_{[\sigma , \omega]^f}$.
\end{proof}

Using the previous proposition and corollary, we can now show that $[\sigma , \omega]^f$ is interval-finite, and this will help us prove that it is $\nn$-graded. 

\begin{corollary}\label{cor: delta f in S infinity is interval-finite}
    Given any $\sigma , \omega \in S_\infty$ with $\sigma <_\text{Bruhat} \omega$, the poset $([\sigma , \omega]^f, <_\text{Bruhat})$ is interval-finite. 
\end{corollary}

\begin{proof}
    Suppose $\mu , \nu \in [\sigma , \omega]^f$ are such that $\mu <_\text{Bruhat} \nu$. By Corollary \ref{cor: big interval contained in ee interval} we have that $[\mu , \nu]_{S_\infty} = [\mu , \nu]_{[\sigma , \omega]^f}$, hence it suffices to show that the former is finite. Let $\tau \in [\mu , \nu]_{S_\infty}$. Since $\mu$ and $\nu$ are both eventually equal to $\sigma$, they are eventually equal to each other. Thus there exists some $i_0 \in \nn$ such that $\mu(i) = \nu(i)$ for all $i \geq i_0$. This implies that the sets $\{ \mu(1) , \ldots, \mu(i_0  - 1) \}$ and $\{ \nu(1) , \ldots, \nu(i_0  - 1) \}$ must be equal. Since $\mu <_\text{Bruhat} \tau <_\text{Bruhat} \nu$, by definition of the Bruhat order, we must have $\{ \mu(1) , \ldots, \mu(i_0  - 1) \} \leq \{ \tau(1) , \ldots, \tau(i_0  - 1) \} \leq \{ \nu(1) , \ldots, \nu(i_0  - 1) \}$. But since the first and last sets are equal, the middle set must be equal to the first (and last) as well. The same result is also clearly true if $i_0$ is replaced by $i_0 + n$ for any $n \in \nn$. Hence by induction on $n$, it must be that $\tau(i_0 + n) = \mu(i_0 + n) = \nu(i_0 + n)$ for all $n \in \nn$. But this implies that $\tau(i) = \mu(i) = \nu(i)$ for all $i \geq i_0$. Since $\tau \in [\mu , \nu]_{S_\infty}$ was arbitrary, there are at most $(i_0  - 1)!$ elements of $[\mu , \nu]_{S_\infty}$. 
\end{proof}

In Proposition \ref{prop: existence of a discrete-saturated chain} we showed that given $\sigma <_\text{Bruhat} \omega$, there exists a discrete-saturated starting at $\sigma$ which is either eventually equal to $\omega$ or converges to $\omega$. Using that $[\sigma , \omega]^f$ is interval-finite, we can now describe when each of these occurs. 

\begin{proposition}\label{prop: when a discrete-saturated chain is finite}
    Let $\sigma , \omega \in S_\infty$ be such that $\sigma <_\text{Bruhat} \omega$. Then $\sigma \approx \omega$ if and only if there exists a finite saturated chain $\sigma  = \sigma_0 \lessdot_\text{Bruhat} \sigma_1 \lessdot_\text{Bruhat} \ldots \sigma_{n - 1} \lessdot_\text{Bruhat} \sigma_n = \omega$.
\end{proposition}

\begin{proof}
If such a discrete-saturated chain exists, by Corollary \ref{cor: cover relations in S infinity are given by transpositions} each cover relation is given by composing with a transposition, hence $\omega$ can be obtained from $\sigma$ by composing with finitely many transpositions, so clearly $\sigma \approx \omega$.
Conversely, suppose $\sigma \approx \omega$ and such a finite chain does not exist. Then by Proposition \ref{prop: existence of a discrete-saturated chain} there exists an infinite discrete-saturated chain $\sigma  = \sigma_0 \lessdot_\text{Bruhat} \sigma_1 \lessdot_\text{Bruhat} \ldots \leq_\text{Bruhat} \omega$ which converges to $\omega$. However since $\sigma$ and $\omega$ are eventually equal, by Proposition \ref{prop: ee implies intervals are equal} we have that $[\sigma , \omega]^f = [\sigma , \omega]$. This interval is finite by Corollary \ref{cor: delta f in S infinity is interval-finite}, but also contains $\sigma_n$ for all $n \in \nn$, which are all distinct, contradiction. 
\end{proof}

We are now able to show that, like $S_\infty^f$, $[\sigma , \omega]^f$ with the Bruhat order is $\nn$-graded, with rank function given by length relative to $\sigma$. 

\begin{proposition}\label{prop: bruhat interval is graded}
    Suppose $\sigma , \omega \in S_\infty$ are such that $\sigma <_\text{Bruhat} \omega$. The poset $([\sigma , \omega]^f, <_\text{Bruhat})$ is $\nn$-graded and the corresponding rank function is the relative length function $\ell_\sigma$.   
\end{proposition}

\begin{proof}
Clearly $\sigma$ is the unique minimal element of $[\sigma , \omega]^f$, so (G1) is satisfied. Given any $\nu \in [\sigma , \omega]^f$, by Propositions \ref{prop: existence of a discrete-saturated chain} and \ref{prop: when a discrete-saturated chain is finite} there exists a finite saturated chain starting at $\sigma$ and ending at $\nu$ in $S_\infty$. This chain is also in $[\sigma , \omega]^f$ by Corollary \ref{cor: big interval contained in ee interval}, so (G2) is also satisfied. Because $[\sigma , \omega]^f$ is interval-finite by Corollary \ref{cor: delta f in S infinity is interval-finite}, all maximal chains of $[\sigma , \nu]$ must be finite, i.e. of the form $\sigma = \sigma_0 \lessdot_{\text{Bruhat}} \sigma_1 \lessdot_\text{Bruhat} \ldots \lessdot_\text{Bruhat} \sigma_m = \nu$. By Lemmas \ref{lem: relative length is additive} and \ref{lem: cover relations increase relative length by one}, we have that $m = \ell_\sigma(\nu)$. Therefore all such maximal chains have the same length, so (G3) is satisfied too.
\end{proof}

We conclude this section with a discussion of maximal and discrete-saturated chains in various subsets of $S_\infty$. This content is not needed for our main result, but is perhaps of interest. As remarked above, in any poset, maximal chains are automatically saturated, but of course the converse does not hold in general. However in (finite) graded posets, saturated chains which include the minimal and maximal elements are maximal chains. A sort of analog of this observation holds for $\nn$-graded posets as well, as the following lemma shows. 

\begin{lemma}\label{lem: discrete-saturated chains are maximal in graded posets}
    If $(P, <)$ is an $\nn$-graded poset, then every discrete-saturated chain $C$ which includes the minimal element $\hat{0}$ is maximal.
\end{lemma}

\begin{proof}
    Let the chain be denoted by $C: \hat{0} = p_0 \lessdot p_1 \lessdot \ldots$ so that if $C$ is not maximal, then there exists some $p \in P \smallsetminus C$ such that $\{ p, p_i \mid 0 \leq i < \infty \}$ is a chain. We must have $p_0 < p$ since $p_0 = \hat{0}$ is the unique minimal element of $P$ by (G1). Suppose $p < p_i$ for some $i \in \nn$. We may assume that $i$ is minimal with this property, so that $p \not< p_{i - 1}$. Since $C$ is a chain, it must be that $p > p_{i - 1}$, which implies $p_{i - 1} < p < p_i$ contradicting that $p_{i - 1} \lessdot p_i$ is a cover. Thus we must have $p_i < p$ for all $i$, so there is a maximal chain in $[\hat{0} , p]$ with infinitely many elements, contradicting (G3).
\end{proof}

\begin{corollary}
$(S_\infty, <_\text{Bruhat})$ is not an $\nn$-graded poset.
\end{corollary}

\begin{proof}
Let $\sigma , \omega \in S_\infty$ be such that $\sigma <_\text{Bruhat} \omega$ and $\sigma \not\approx \omega$. By Propositions \ref{prop: existence of a discrete-saturated chain} and \ref{prop: when a discrete-saturated chain is finite} there exists an infinite discrete-saturated chain $\sigma = \sigma_0 \lessdot_\text{Bruhat} \sigma_1 \lessdot_\text{Bruhat} \ldots \leq_\text{Bruhat} \omega$ in $S_\infty$ converging to $\omega$. However by Corollary \ref{cor: no maximal elements} there exists some transposition $(p , q)$ such that $\omega \lessdot_\text{Bruhat} \omega \circ (p , q)$, so in particular our chain is not maximal. Thus the desired statement follows from Lemma \ref{lem: discrete-saturated chains are maximal in graded posets}. 
\end{proof} 

\begin{remark}
    Even though the poset $[\sigma , \omega]^f$ is nicer than $S_\infty$, in the sense that it is $\nn$-graded, it still does not have a maximal element, as the following corollary shows.
\end{remark}

\begin{corollary}\label{cor: no max element for intervals either}
Given $\sigma , \omega \in S_\infty$ with $\sigma <_\text{Bruhat} \omega$ and $\sigma \not\approx \omega$, the poset $([\sigma , \omega]^f, <_\text{Bruhat})$ does not have a maximal element. 
\end{corollary}

\begin{proof}
Given any $\nu \in [\sigma , \omega]^f$ note that we cannot have $\nu \approx \omega$ since we know that $\nu \approx \sigma$ and $\sigma \not\approx \omega$ and ``$\approx$'' is an equivalence relation. Thus by Propositions \ref{prop: existence of a discrete-saturated chain} and \ref{prop: when a discrete-saturated chain is finite} there exists an infinite discrete-saturated chain $\nu = \nu_0 \lessdot_\text{Bruhat} \nu_1 \lessdot_\text{Bruhat} \ldots$ in $[\nu , \omega]^f \subseteq [\sigma , \omega]^f$, and in particular $\nu$ is not maximal. 
\end{proof}


\section{Lexicographically Nested-Shellable Intervals of $S_\infty$}\label{sec: Lexicographically Nested-Shellable Intervals of S infty}

This section is dedicated to proving Theorem \ref{thm: main thm Bruhat order} which says that $[\sigma , \omega]^f$ is a lexicographically nested-shellable poset. We begin with a lemma which will imply (LNS2).

\begin{lemma}\label{lem: element always less than some nu n}
Given $\sigma , \omega \in S_\infty$ with $\sigma <_\text{Bruhat} \omega$ and $\sigma \not\approx \omega$, and a discrete-saturated chain $C: \sigma = \sigma_0 \lessdot_\text{Bruhat} \sigma_1 \lessdot_\text{Bruhat} \ldots \leq_\text{Bruhat} \omega$ in $[\sigma , \omega]^f$ which converges to $\omega$, if $\nu \in [\sigma , \omega]^f$, then there exists $n \in \nn$ such that $\nu \leq_\text{Bruhat} \sigma_n$. 
\end{lemma}

\begin{proof}
    If $\nu \in [\sigma , \omega]^f$, since $\nu \approx \sigma$, there exists some $i_0 \in \nn$ such that ($\ast$) $\sigma(i) = \nu(i)$ for all $i \geq i_0$. Furthermore, since $\sigma_0 , \sigma_1, \ldots$ converges to $\omega$, there is some $N \in \nn$ such that if $n \geq N$ then ($\ast \ast$) $\sigma_n(i) =  \omega(i)$ for all $i \leq i_0$. We shall appeal to the tableau criterion of the Bruhat order to show that $\nu \leq_\text{Bruhat} \sigma_n$. Indeed if $j \leq i_0$, then $\{ \sigma_n(1) , \ldots , \sigma_n(j) \} = \{ \omega(1) , \ldots, \omega(j) \} \geq \{ \nu(1) , \ldots , \nu(j) \}$ where the equality holds by ($\ast \ast$) and the inequality holds because $\nu \leq_\text{Bruhat} \omega$. On the other hand, if $j \geq i_0$, then $\{ \nu(1), \ldots, \nu(j) \} = \{ \sigma(1), \ldots, \sigma(j) \} \leq \{\sigma_n(1) , \ldots, \sigma_n(j) \}$, where the inequality holds because $\sigma \leq_\text{Bruhat} \sigma_n$ and the equality holds because by ($\ast$), $\nu$ and $\sigma$ agree past $i_0$. So indeed the tableau criterion shows that $\nu \leq_\text{Bruhat} \sigma_n$.
\end{proof}

Now we introduce the relevant edge-labeling of $(S_\infty , <_\text{Bruhat})$. Let $T_\infty$ denote the set of transpositions in $S_\infty$ and recall that every such transposition can be written uniquely in cycle notation $(p , q)$ where $p < q$. Like Edelman in \cite{edelman1981bruhat} we consider the \emph{lexicographic order} on $T_\infty$, and denote it by $<_\text{lex}$. Explicitly, if $(p , q)$ and $(p' , q')$ are two transpositions (with $p < q$ and $p' < q'$) then $(p , q) <_\text{lex} (p', q')$ if and only if either $p < p'$ or $p = p'$ and $q < q'$. Notice that $<_\text{lex}$ is a total order (with order-type the ordinal $\omega^2$). Define an edge-labeling $t_\text{lex} : C(S_\infty , <_\text{Bruhat}) \to (T_\infty , <_\text{lex})$ as follows. If $\sigma \lessdot_\text{Bruhat} \nu$, then by Corollary \ref{cor: cover relations in S infinity are given by transpositions} there exists a unique transposition $(p , q) \in T_\infty$ such that $\sigma \circ (p , q) = \nu$. We define $t_\text{lex}(\sigma , \nu) = (p , q)$. 

\begin{remark}
    As in $S_n$, it is also true that for any cover relation $\sigma \lessdot_\text{Bruhat} \nu$ in $S_\infty$ there exists a unique transposition $(p', q')$ such that $(p' , q') \circ \sigma = \nu$, so we could (as Edelman does) define the edge labeling using this convention. However we have chosen our definition of $t_\text{lex}$ because the proofs in our infinite setting are somewhat easier in this case. 
\end{remark}

\begin{remark}
    The remaining proofs are very similar in spirit to those Edelman uses in \cite{edelman1981bruhat} to prove that $S_n$ is lexicographically shellable. We repeat them here to emphasize that they still work in $S_\infty$ with the combinatorial machinery we have built up. 
\end{remark}

In order to show that $t_\text{lex}$ is an $L$-labeling, we begin with a lemma about how repeatedly applying transpositions of the form $(d , m)$ behaves with respect to the lex order on transpositions.

\begin{lemma}\label{lem: d m increasing in lex ordering}
    Let $\sigma , \omega \in s_\infty$ be such that $\sigma <_\text{Bruhat} \omega$. Define $\sigma_1 := \sigma \circ (d(\sigma , \omega) , m(\sigma , \omega))$. Then $(d(\sigma, \omega), m(\sigma,\omega)) <_\text{lex} (d(\sigma_1, \omega), m(\sigma_1,\omega))$ 
\end{lemma}

\begin{proof}
First of all, we have that $\sigma_1(n) = \sigma(n) = \omega(n)$ for all $n < d(\sigma , \omega)$, hence $d(\sigma , \omega) \leq d(\sigma_1 , \omega)$. If $d(\sigma , \omega) < d(\sigma_1 , \omega)$, then we certainly have $(d(\sigma, \omega), m(\sigma,\omega)) <_\text{lex} (d(\sigma_1, \omega), m(\sigma_1,\omega))$, so suppose that $d(\sigma , \omega) = d(\sigma_1 , \omega)$ and suppose for contradiction that $m( \sigma , \omega ) \geq m( \sigma_1 , \omega)$. First of all, if we have $m( \sigma , \omega ) = m( \sigma_1 , \omega)$, then we compute: 
\begin{equation*}
    \sigma_1(d(\sigma_1 , \omega)) \overset{(1)}{=} \sigma_1(d(\sigma, \omega)) \overset{(2)}{=} \sigma(m(\sigma, \omega)) \overset{(3)}{>} \sigma(d(\sigma , \omega)) \overset{(4)}{=} \sigma_1(m(\sigma, \omega)) \overset{(5)}{=} \sigma_1(m(\sigma_1 , \omega))
\end{equation*}

Here (1) follows because $d(\sigma_1 , \omega) = d(\sigma, \omega)$ by assumption, (2) and (4) follow from the definition of $\sigma_1$, (3) follows from the definition of $d$ and $m$, and (5) follows because $m(\sigma_1 , \omega) = m(\sigma, \omega)$. However this contradicts that $\sigma_1(d(\sigma_1 , \omega) ) < \sigma_1(m(\sigma_1 , \omega))$, which again follows from the definition of $d$ and $m$. 

Now suppose $m(\sigma_1 , \omega) < m(\sigma , \omega)$. We compute:

\begin{align*}
    \sigma(d(\sigma , \omega)) \overset{(1)}{<} \sigma(m(\sigma , \omega)) 
    \overset{(2)}{=} \sigma_1(d(\sigma , \omega)) 
    \overset{(3)}{=} \sigma_1(d(\sigma_1 , \omega)) 
\overset{(4)}{<} \sigma_1(m(\sigma_1 , \omega)) 
\overset{(5)}{<} \sigma_1(f(\sigma_1 , \omega)) 
&\overset{(6)}{=} \sigma(f(\sigma_1 , \omega)) 
\\&\overset{(7)}{=} \sigma(f(\sigma , \omega))
\end{align*}

Here (1), (4), and (5) follow from the definition of $d$, $m$, and $f$, (2) follows from the definition of $\sigma_1$, (3) follows because $d(\sigma_1 , \omega) = d(\sigma, \omega)$ by assumption, and (6) and (7) follow because $\sigma$ and $\sigma_1$ only differ on $d(\sigma , \omega)$ and $m(\sigma , \omega)$ and since we assumed that $d(\sigma , \omega) = d(\sigma_1 , \omega)$, it must be that $m(\sigma , \omega) \neq f(\sigma , \omega)$ and therefore $f(\sigma , \omega) = f(\sigma_1 , \omega)$ and that $\sigma(f(\sigma_1 , \omega)) = \sigma_1(f(\sigma_1 , \omega))$. All together this shows that $\sigma(d(\sigma , \omega)) < \sigma_1(m(\sigma_1 , \omega)) < \sigma(f(\sigma , \omega))$, however again since $\sigma$ and $\sigma_1$ only differ on $d(\sigma , \omega)$ and $m(\sigma , \omega)$ and $d(\sigma , \omega) = d(\sigma_1 , \omega)  < m(\sigma_1 , \omega) < m(\sigma , \omega)$ by assumption, we have $\sigma(m(\sigma_1 , \omega)) = \sigma_1(m(\sigma_1 , \omega))$, and hence we obtain $\sigma(d(\sigma , \omega)) < \sigma(m(\sigma_1 , \omega)) < \sigma(f(\sigma , \omega))$, which contradicts the minimality of $m(\sigma , \omega)$. 
\end{proof}

As a corollary, it follows immediately that the discrete-saturated chains of Proposition \ref{prop: existence of a discrete-saturated chain} obtained by repeatedly applying transpositions of the form $(d , m)$ are lexicographically increasing. 

\begin{corollary}\label{cor: existence of an increasing discrete-saturated chain}
    Let $\sigma , \omega \in S_\infty$ be such that $\sigma <_\text{Bruhat} \omega$. The Jordan-H\"{o}lder sequence of the (finite or infinite) discrete-saturated chain $\sigma  = \sigma_0 \lessdot_\text{Bruhat} \sigma_1 \lessdot_\text{Bruhat} \ldots \leq_{\text{Bruhat}} \omega$ defined by $\sigma_i = \sigma_{i - 1} \circ (d(\sigma_{i - 1} , \omega) , m(\sigma_{i - 1} , \omega) )$ (which is guaranteed to exist by Proposition \ref{prop: existence of a discrete-saturated chain}) is $t_\text{lex}$-increasing. 
\end{corollary}

Before we come to our main result, we prove one final necessary lemma about how the relative candidate $(d, m)$ of $\sigma , \omega$ compares to other relative candidates in the lexicographic ordering. 

\begin{lemma}\label{lem: d , m is lex first relative candidate}
    Suppose $\sigma, \omega \in S_\infty$ are such that $\sigma <_\text{Bruhat} \omega$. Then $(d(\sigma , \omega) , m(\sigma , \omega))$ is the lexicographically first relative candidate for $\sigma , \omega$. 
\end{lemma}

\begin{proof}
    Let $d = d(\sigma , \omega)$, $m = m(\sigma , \omega)$, and $f = f(\sigma , \omega)$ and suppose $(p , q)$ is a relative candidate for $\sigma , \omega$ which is not equal to $(d , m)$. Suppose for contradiction that $p < d$. Then by definition of $d$, we have $\sigma(p) = \omega(p)$, but by (C1) we have that $\sigma(p) < \sigma(q)$, hence $\omega(p) < \sigma(q) = (\sigma \circ (p , q))(p)$. However we have that $(\sigma \circ (p , q))(n) = \omega(n)$ for all $n < p$. Thus $d(\sigma \circ (p , q) , \omega) = p$ but by Lemma \ref{lem: less than bruhat implies values at d less than} this contradicts that $\sigma \circ (p , q) <_\text{Bruhat} \omega$. Thus we must have $d \leq p$. If $d < p$ then $(d , m) <_\text{lex} (p , q)$ and we are done. So suppose that $d = p$ and again for contradiction suppose $q < m$. Then in particular we have $d = p < q < m < f$, so by definition of $d$ and $f$ and by minimality of $m$, it must be that $\sigma(q) < \sigma(d) = \sigma(p)$ or $\sigma(q) > \sigma(f)$, however $(p , q)$ is a relative candidate for $\sigma, \omega$ so by (C1) we have $\sigma(p) < \sigma(q)$, hence it must be that $\sigma(q) > \sigma(f)$. But then $(\sigma \circ (p , q))(p) = \sigma(q) > \sigma(f) = \omega(p)$. Since $(\sigma \circ (p , q))(n) = \omega(n)$ for all $n < p$, again by Lemma \ref{lem: less than bruhat implies values at d less than}, this contradicts that $\sigma \circ (p , q) <_\text{Bruhat} \omega$. Hence if $d = p$ then we must have $m \leq q$, but if $m = q$ then $(d , m) = (p , q)$, so we must have $m < q$ and hence $(d , m) <_\text{lex} (p , q)$ as desired. 
\end{proof}

We can now prove our main result.

\begin{proof}[Proof of Theorem \ref{thm: main thm Bruhat order}]
By Corollary \ref{cor: delta f in S infinity is interval-finite}, $[\sigma , \omega]^f$ is interval-finite, and by Proposition \ref{prop: bruhat interval is graded} it is $\nn$-graded. By Proposition \ref{prop: existence of a discrete-saturated chain} there exists a discrete-saturated chain $\sigma = \sigma_0 \lessdot_\text{Bruhat} \sigma_1 \lessdot_\text{Bruhat} \ldots \leq_\text{Bruhat} \omega$ which is either equal to $\omega$ or converges to $\omega$. In either case, Lemma \ref{lem: element always less than some nu n} implies that (LNS2) holds. 

To show (LNS1), let $\mu , \nu \in [\sigma , \omega]^f$ be such that $\mu <_\text{Bruhat} \nu$. We must show that $t_\text{lex}$ is an $L$-labeling for $[ \mu , \nu]$. First note that by Corollary \ref{cor: existence of an increasing discrete-saturated chain} there exists a finite saturated chain $C: \mu = \mu_0 \lessdot_\text{Bruhat} \mu_1 \lessdot_\text{Bruhat} \ldots \lessdot_\text{Bruhat} \mu_n = \nu$ whose Jordan-H\"{o}lder sequence is $t_\text{lex}$-increasing (so the existence part of (L1) is satisfied), and furthermore at each stage we have $\mu_{i + 1} = \mu_i \circ (d(\mu_i , \nu), m(\mu_i , \nu))$. Suppose that $D: \mu_0 \lessdot_\text{Bruhat} \mu_1' \lessdot_\text{Bruhat} \ldots \lessdot_\text{Bruhat} \mu_n = \nu$ is another discrete-saturated chain in $[\mu , \nu]$ which is not equal to $C$. Then there exists some minimal $i$ such that $\mu_i = \mu_i'$ but $\mu_{i + 1} \neq \mu_{i + 1}'$. Let $(p , q) = t_\text{lex}(\mu_i , \mu_{i + 1}')$. Then both $(d (\mu_i , \nu) , m(\mu_i , \nu))$ and $(p , q)$ are relative candidates for $\mu_i , \nu$, which are not equal, since $\mu_{i + 1} \neq \mu_{i + 1}'$. So by Lemma \ref{lem: d , m is lex first relative candidate} we must have $(d (\mu_i , \nu), m(\mu_i , \nu))  <_\text{lex} (p , q)$, which implies that the Jordan-H\"{o}lder sequence of $C$ precedes that of $D$ in the lexicographic ordering, thus (L2) is satisfied.  

Now we must show the uniqueness part of (L1). Suppose that the chain $D$ above is also $t_\text{lex}$-increasing. Since $(d , m) := (d (\mu_i , \nu), m(\mu_i , \nu))  <_\text{lex} (p , q)$, either $d < p$ or $d = p$ and $m < q$. If the former case holds, then since $D$ is $t_\text{lex}$-increasing, every transposition $(p' , q')$ appearing in the Jordan-H\"{o}lder sequence of $D$ after $(p , q)$ satisfies $p' \geq p$, so for all $i \leq j \leq n$ we have $\mu_j(d) = \mu_i(d) \neq \nu(d)$. In particular taking $j = n$ we obtain a contradiction. Now suppose the latter case holds. We consider two sub-cases. On the one hand, if $q > f := f(\mu_i , \nu)$, then since every transposition $(p' , q')$ appearing in the Jordan-H\"{o}lder sequence of $D$ after $(p , q)$ is lexicographically larger than $(p , q)$, it follows for all $i \leq j \leq n$ that $\mu_j(d) \neq \mu_i(f) = \nu(d)$ and in particular taking $j = n$ yields a contradiction. 

On the other hand, if $q \leq f$, then since $(p , q)$ is a relative candidate for $\mu_i , \nu$, we must have $\mu_i(p) < \mu_i(q)$, and that either $\mu_i(m) < \mu_i(p) = \mu_i(d)$ or $\mu_i(m) > \mu_i(q)$. However we know that $(d , m)$ is a relative candidate for $\mu_i , \nu$, so the first inequality cannot happen, and we must have $\mu_i(m) > \mu_i(q)$. Notice that because $\mu_n = \nu$, and the transpositions in the Jordan-H\"{o}lder sequence of $D$ increase in the lexicographic ordering, the transposition $(d , f)$ must appear at some point in this sequence in order to obtain a permutation which satisfies $\nu(d) = \mu_i(f)$. Hence part of the Jordan-H\"{o}lder sequence of $D$ contains the transpositions $(p , q) = (p_i , q_i) <_\text{lex} (p_{i + 1}, q_{i + 1}) <_\text{lex} \ldots <_\text{lex} (p_\ell , q_\ell) = (d , f)$. Since $p = d$, we must have $p_j = d$ for all $j$, and $m < q < q_i < \ldots < q_\ell = f$. We now show that $\mu_j(q_j) < \mu_j(m)$ for all $i \leq j \leq \ell$ by induction on $j$. Indeed, we already showed the base case, so suppose it holds for some $j$. Then $(d , q_{j + 1})$ is a relative candidate for $\mu_{j + 1} , \nu$, thus we must have $\mu_{j + 1}(d) < \mu_{j + 1}(q_{j + 1})$ and either $\mu_{j + 1}(m) < \mu_{j + 1}(d)$ or $\mu_{j + 1}(m) > \mu_{j + 1}(q_{j + 1})$. However $\mu_{j + 1}(d) = \mu_j(q_j)$, and $\mu_j(q_j) < \mu_j(m) = \mu_{j + 1}(m)$ by the induction hypothesis, so we must have $\mu_{j + 1}(m) > \mu_{j + 1}(q_{j + 1})$ as desired. Therefore in particular we have $\mu_i(f) = \mu_\ell(f) = \mu_\ell(q_\ell) < \mu_\ell(m) = \mu_i(m)$, which contradicts the definition of $m$ and $f$. Therefore $C$ is the unique $t_\text{lex}$-increasing discrete-saturated chain in $[\mu , \nu]$, so (L1) is satisfied as well. 
\end{proof}



\bibliographystyle{alpha}

\bibliography{references}

\end{document}